\documentclass[preprint,12pt]{elsarticle}

\usepackage{amsfonts}
\usepackage{xcolor}
\usepackage{amssymb}
\usepackage{comment} 
\usepackage{upgreek}
\usepackage{amsmath}
\usepackage{graphicx}
\usepackage{tikz-cd}
\usepackage{amscd}
\usepackage{endnotes}
\usepackage{tikz}
\usepackage{amsthm,mathtools}
\usepackage[colorlinks=true, linkcolor=blue, citecolor=blue, urlcolor=blue]{hyperref}

\usetikzlibrary{positioning}

\newtheorem{theo}{Theorem}[section]
\newtheorem{cor}[theo]{Corollary}
\newtheorem{lem}[theo]{Lemma}
\newtheorem{prop}[theo]{Proposition}

\theoremstyle{definition}
\newtheorem{defn}[theo]{Definition}

\theoremstyle{remark}
\newtheorem{rmk}[theo]{Remark}
\newtheorem{ex}[theo]{Example}

\numberwithin{equation}{section}

\newcommand{\Q}{\mathcal{Q}_2}
\newcommand{\D}{\mathcal{D}_2}
\newcommand{\F}{\mathcal{F}_2}
\newcommand{\C}{\mathcal{O}_2}
\newcommand{\BB}{\mathcal{B}_2}
\newcommand{\M}{M_{2^\infty}(\mathbb{C})}
\newcommand{\B}{M_{2^\infty}(C(\mathbb{T}))}
\newcommand{\K}{\mathcal{K}} 
\newcommand{\z}{\mathbf{Z}_2}

\journal{}
\makeatletter
\def\ps@pprintTitle{%
  \let\@oddhead\@empty
  \let\@evenhead\@empty
  \let\@oddfoot\@empty
  \let\@evenfoot\@empty
}
\makeatother

\begin{document}

\begin{frontmatter}



\title{Maximality and symmetry related to \\ the \(2\)-adic ring \(C^*\)-algebra}

\author[a]{Dolapo Oyetunbi}
 \affiliation[a]{organization={Department of Mathematics and Statistics},
            addressline={ University of Windsor},
             postcode={ ON N9B 3P4},
            country={Canada}\\
          }
         \ead{oyetunb@uwindsor.ca}

\author[a]{Dilian Yang}
\ead{dyang@uwindsor.ca}

\begin{abstract}
The 2-adic ring $C^*$-algebra $\mathcal{Q}_2$ is the universal $C^*$-algebra generated by a unitary and an isometry satisfying certain relations. It contains a canonical copy of the Cuntz algebra $\mathcal{O}_2$. We show that $\mathcal{O}_2$ is a maximal $C^*$-subalgebra of $\mathcal{Q}_2$. Furthermore, we examine the structure of the fixed-point algebra under a periodic \(^*\)-automorphism $\sigma$ of $\mathcal{Q}_2$, which is extended from the flip-flop \(^*\)-automorphism of $\mathcal{O}_2$. We show that the maximality of $\mathcal{O}_2$ in $\mathcal{Q}_2$ extends to the crossed product $\mathcal{O}_2 \rtimes_{\sigma} \mathbb{Z}_2$ in $\mathcal{Q}_2 \rtimes_{\sigma} \mathbb{Z}_2$, and to the fixed-point algebra $\mathcal{O}_2^\sigma$ in $\mathcal{Q}_2^\sigma$. As a consequences of our main results, a few open questions concerning $\mathcal{Q}_2$ are resolved.

\end{abstract}




\end{frontmatter}



\section{Introduction}
Over the years, the structure of intermediate operator algebras arising from inclusions $A \subset B$ has played a significant role in the classification of operator algebras and the study of rigidity phenomena. Notably, such algebras feature prominently in Popa’s deformation/rigidity theory \cite{Popa2006, Popa2007} and in the characterization of \(C^*\)-simplicity via actions on certain boundaries \cite{KK17}. One approach to analyze these inclusions involves Galois-type correspondences \cite{BEFRPR21, Izumi2002, Suzuki2020} and crossed products by partial subactions \cite{KD24}. Tensor product constructions likewise yield a broad class of examples \cite{Ge-Kadison96, Zacharias2001}.

Group actions and their fixed-point algebras often encode rich algebraic and dynamical information. An interesting problem is to understand how the global structure of a \(C^*\)-algebra is preserved or breaks down under its symmetries (\(^*\)-automorphism of order 2). For instance, Blackadar constructed a symmetry of the UHF \(C^*\)-algebra of type \( 2^\infty \) whose fixed-point algebra is not approximately finite-dimensional \cite{Blackadar1990}. Meanwhile, certain structural properties, such as nuclearity, are known to be preserved when passing to fixed-point algebras and crossed products under finite group actions \cite{brown08}.

The 2-adic ring \(C^*\)-algebra $\mathcal{Q}_2$ is the universal \(C^*\)-algebra generated by a unitary $u$ and an isometry $s_2$ satisfying some relations. 
$\mathcal{Q}_2$ can be viewed as a symmetrized version of the Cuntz algebra $\mathcal{O}_2$, containing a canonical copy of $\mathcal{O}_2=C^*(s_1, s_2)$ with $s_1:=us_2$ \cite{LL12}. Unlike $\mathcal{O}_2$, the unitary $u$ intertwines the Cuntz isometries $s_1$ and $s_2$ in a way that precludes the existence of a conditional expectation from $\mathcal{Q}_2$ onto $\mathcal{O}_2$ \cite{ACR18}. This intertwining imposes a rigidity that obstructs many structural phenomena possible in $\mathcal{O}_2$ from occurring in $\mathcal{Q}_2$. The algebra $\mathcal{Q}_2$ can be realized as the $C^*$-algebra of a self-similar graph \cite{ES16,VY25}, as a partial crossed product \cite{ES16}, as a boundary quotient $C^*$-algebra of a semigroup \cite{BOS18, LY21}, and as a $C^*$-algebra associated to a group endomorphism of a discrete group \cite{Hirshberg2002}. These various realizations provide invaluable tools to analyze $\mathcal{Q}_2$, making it a natural test case for exploring structural properties that may extend to broader classes of $C^*$-algebras. The $C^*$-algebra $\mathcal{Q}_2$ has been extensively studied in the literature; see, for instance, \cite{ACR18b, ACR18, ACR20, ACR21, LL12} and the references therein.

Very recently, Bassi and Conti showed that the inclusion $\mathcal{O}_2 \subset \mathcal{Q}_2$ is rigid \cite{BC2025}. However, the natural question of whether there exists a proper intermediate $C^*$-subalgebra between $\mathcal{O}_2$ and $\mathcal{Q}_2$ remains open. At first glance, the question may appear straightforward, and a natural candidate is the \(C^*\)-algebra $C^*(\mathcal{O}_2, u^n)$ (see Definition \ref{defn:2-adic}) for some $n\in \mathbb{N}$. However, the intertwining relations imposed on the Cuntz isometries imply that $C^*(\mathcal{O}_2, u^n) = \mathcal{Q}_2$ for all $n\in \mathbb{Z}$ (see \cite[Remark 9.2]{ACR21}). We resolve the question by establishing the following:

\begin{theo}[Theorem \ref{thm:maximality_of_cuntz}]
\label{T:9.1}
$\mathcal{O}_2$ is a maximal $C^*$-subalgebra of $\mathcal{Q}_2$.
\end{theo}

The flip-flop $^*$-automorphism $\sigma: \mathcal{O}_2 \to \mathcal{O}_2$ is a symmetry of the Cuntz algebra that exchanges its two canonical generating isometries $s_1$ and $s_2$. Choi and Latr\'emoli\`ere gave a concrete description of the fixed-point algebra under this \(^*\)-automorphism and showed that it is well-behaved: both the associated fixed-point algebra and crossed product are isomorphic to $\mathcal{O}_2$ \cite[Theorem 1.4]{CL12}. The flip-flop \(^*\)-automorphism $\sigma$ extends naturally to a \(^*\)-automorphism of $\mathcal{Q}_2$ which sends \(u\) to \(u^*\). As one of our main results, we give an explicit description of the corresponding fixed-point algebra and extend the result of Choi and Latr\'emoli\`ere \cite{CL12} (see also \cite[Question 9.13]{ACR21}).

\begin{theo}[Theorem \ref{thm:fixed_sym} and Theorem \ref{T:3iso}]
\label{T:9.13}
Let \(\sigma\) be the \(^*\)-automorphism of \(\mathcal{Q}_2\) extending the flip-flop \(^*\)-automorphism of \(\mathcal{O}_2\),
\[
T \coloneqq \frac{1}{\sqrt{2}}(s_1 + s_2), \quad
V \coloneqq \frac{1}{\sqrt{2}}(s_1 - s_2)(s_1 s_1^* - s_2 s_2^*), \quad \text{and} \quad
R \coloneqq u + u^*.
\]
Then the fixed-point algebra of \(\sigma\) is
\[
\mathcal{Q}_2^\sigma = C^*(T, V, R) = C^*\big(x u^h + \sigma(x) u^{-h} \mid x \in \mathcal{O}_2,\, h \in \mathbb{Z} \big).
\]
Moreover, \(\mathcal{Q}_2 \rtimes_{\sigma} \mathbb{Z}_2 \cong \mathcal{Q}_2^\sigma \cong \mathcal{Q}_2\).
\end{theo}

We further demonstrate that the maximality of the Cuntz algebra $\C$ inside $\Q$ extends to the setting of the crossed product and fixed-point algebras associated with the symmetry $\sigma$. Specifically, we establish the following:

\begin{theo}[Theorem \ref{thm:maximality_car_sym} and Corollary \ref{cor:maximality_cuntz_fixed}]
 $\C \rtimes_{\sigma} \mathbb{Z}_2$ is a maximal $C^*$-subalgebra of $\Q \rtimes_{\sigma} \mathbb{Z}_2$; and $\mathcal{O}_2^\sigma$ is a maximal $C^*$-subalgebra of $\Q^\sigma$.
\end{theo}

A key ingredient for establishing the maximality results above is the following:
\begin{theo}[Theorem \ref{thm:maximality_of_car}]
\(\M\) is a maximal \( C^* \)-subalgebra of \( \B \).
\end{theo}

Our approach to prove the results is to realize the relevant \( C^* \)-algebras as corners of partial crossed products and to analyze the structure of these crossed products by examining the properties of the associated automorphisms. We show that the \(^*\)-automorphism implementing the crossed product satisfies a suitable version of finite Rokhlin dimension, as introduced in \cite{AHS23}. We then apply \cite{AHS23} to conclude that every intermediate $C^*$-algebra between two crossed products $A \rtimes_{\alpha,r} G$ and $B \rtimes_{\alpha,r} G$ is itself a reduced crossed product $C \rtimes_{\alpha,r} G$ for some $C$ with $A \subset C \subset B$.

Our results resolve a number of open questions concerning the structure of $\mathcal{Q}_2$ posed in the survey \cite{ACR21}. In particular, we establish the maximality of $\mathcal{O}_2$ in $\mathcal{Q}_2$, analyze the structure of the unitary normalizer of $\mathcal{O}_2$ in $\mathcal{Q}_2$, provide a complete description of the fixed-point algebra $\mathcal{Q}_2^\sigma$, 
and prove that \(\mathcal{Q}_2^\sigma \cong \mathcal{Q}_2\cong \mathcal{Q}_2 \rtimes_{\sigma} \mathbb{Z}_2\).

\bigskip
\noindent
\textbf{Acknowledgements.} 
This work is partially supported by an NSERC Discovery Grant.

\section{Preliminaries and notations}

In this section, we provide necessary background, including basics on the 2-adic ring $C^*$-algebra, central sequences, partial actions, and Rokhlin dimension. We also set up notation that will be used throughout the rest of the paper.

\subsection{Some basics on $\mathcal{Q}_2$}
\label{SS:Q2basics}
As discussed in the introduction, the 2-adic ring $C^*$-algebra $\Q$ admits several constructions in the literature. This subsection summarizes some fundamental properties of $\Q$. For further details, we refer the reader to \cite{ACR21} and the references therein.

\begin{defn}\label{defn:2-adic}
The \emph{2-adic ring $C^*$-algebra} $\mathcal{Q}_2$ is the universal unital $C^*$-algebra generated by a unitary $u$ and an isometry $s_2$ satisfying the relations:
$$
\text{(I)} \quad s_2 u = u^2 s_2, \qquad
\text{(II)} \quad s_2 s_2^* + u s_2 s_2^* u^{-1} = 1.
$$
\end{defn}

By the universal property of $\Q$, for each $t\in \mathbb{T}$ there is $\alpha_t\in \mathrm{Aut}(\Q)$ defined by 
$
\alpha_t(u) = u\text{ and }\alpha_t(s_2) = t s_2. 
$
This yields an action of $\mathbb{T}$ on $\Q$, which is known as the \textit{gauge action} of $\Q$. 

The $C^*$-algebra $\Q$ contains the following notable $C^*$-subalgebras:
\begin{itemize}
\item[$\blacktriangleright$]
the canonical copy of the Cuntz algebra:
\[
\mathcal{O}_2 \coloneqq C^*(s_1, s_2), \quad \text{where } s_1 := u s_2;
\]

\item[$\blacktriangleright$]
a Cartan $C^*$-subalgebra of $\Q$:
\[
\mathcal{D}_2\coloneqq \overline{\operatorname{span}}\{u^m s_{2^n} s_{2^n}^* u^{-m}, 
m \in \mathbb{Z}, n \in \mathbb{N}\},
\]
where $s_{2^n} := s_2^n$;

\item[$\blacktriangleright$] the fixed-point algebra of the gauge action \(\alpha\):
\begin{equation}\label{eqn:defn_bunce}
\BB \coloneqq \overline{\operatorname{span}} \{ u^m s_{2^n} s_{2^n}^* u^k : m, k \in \mathbb{Z},\; n \in \mathbb{N} \};
\end{equation}
equivalently, \(\BB\) is the \(C^*\)-algebra generated by \(\D\) and \(u\);

\item[$\blacktriangleright$]
the fixed-point algebra of $\alpha|_{\C}$:
\begin{equation}\label{eqn:defn_car}
\F \coloneqq \overline{\operatorname{span}} \{ u^m s_{2^n} s_{2^n}^* u^{-k} : 0 \le m, k \le 2^n -1,\; n \in \mathbb{N} \}.
\end{equation}
\end{itemize}

 Clearly, we have the following inclusion lattice of the above \( C^* \)-subalgebras of 
 \( \Q \):
\[
\begin{tikzcd}[row sep=1.3em, column sep=large]
  \mathcal{Q}_2 \arrow[r, phantom, "\supset"] & 
  \BB \arrow[r, phantom, "\supset"] & 
  C^*(u) \\
  \mathcal{O}_2 \arrow[r, phantom, "\supset"] \arrow[u, phantom, "\cup"] & 
  \F \arrow[r, phantom, "\supset"] \arrow[u, phantom, "\cup"] & 
  \mathcal{D}_2 \\
  C^*(s_2) \arrow[u, phantom, "\cup"] & &
\end{tikzcd}\]

\begin{rmk}\label{rmk:form_adic}
Let us record some properties of the above $C^*$-subalgebras which will be useful later.
\leavevmode
\begin{enumerate}
    \item Repeated use of conditions (I) and (II) from Definition \ref{defn:2-adic} shows that
\begin{align*}
    \Q 
    &= \overline{\operatorname{span}} \left\{ u^{k_1} s_{2^{n_1}} s_{2^{n_2}}^* u^{k_2} : 0 \le k_1 \le 2^{n_1} - 1,\ k_2 \in \mathbb{Z},\ n_i \in \mathbb{N},\ i = 1,2 \right\}\\
    &= C^*(\mathcal{B}_2, u),\text{ and }\\
    \C &= \overline{\operatorname{span}} \left\{ u^{k_1} s_{2^{n_1}} s_{2^{n_2}}^* u^{-k_2} : 0 \le k_i \le 2^{n_i} - 1,\ n_i \in \mathbb{N},\ i = 1,2 \right\}.
\end{align*}

\item The spectrum of \(\mathcal{D}_2\) is homeomorphic to the 2-adic integers \(\z\) (\cite{LL12}).

\item The $C^*$-subalgebra $\mathcal{B}_2$ is a Bunce-Deddens algebra of type $2^\infty$
(\cite{BOS18, VY25}).

\item The $C^*$-subalgebra $\mathcal{F}_2$ is a UHF algebra of type $2^\infty$. 

\item $\Q$ is a Kirchberg algebra satisfying the UCT and \(K_0 (\Q) = K_1 (\Q) = \mathbb{Z}\) (\cite{LL12}). 
\end{enumerate}
\end{rmk}

\subsection{Central sequences}
In this subsection, we briefly recall some basic facts about central sequences (see \cite{Kirch04} for further details).

Let $A$ be a C$^*$-algebra. Define
$A^\infty := \ell^\infty(\mathbb{N}, A) / c_0(\mathbb{N}, A),
$
where $\ell^\infty(\mathbb{N}, A)$ is the C$^*$-algebra of bounded sequences in $A$, and $c_0(\mathbb{N}, A)$ consists of sequences that converge to zero in norm. Denote the canonical quotient map by
$\pi_\infty^A : \ell^\infty(\mathbb{N}, A) \to A^\infty
$. The \(C^*\)-algebra $A$ embeds into $A^\infty$ as the image under $\pi_\infty^A$ of constant sequences in $\ell^\infty(\mathbb{N}, A)$. For convenience, we use the notation $a \in A$ interchangeably with $\pi_\infty^A(a, a, \ldots) \in A^\infty$. The central sequence algebra is defined as
$$
A ^c \coloneqq A^\infty \cap A' 
:= \pi_\infty^A\left( \big\{(a_n) \in \ell^\infty(\mathbb{N}, A) : \lim_{n \to \infty} \|[a_n, a]\| = 0 \text{ for all } a \in A \big\} \right).
$$
The (two-sided) annihilator 
$$
\operatorname{Ann} (A)\coloneqq \operatorname{Ann}(A, A^\infty) := \left\{ b\in A^\infty: bA=0=Ab\right\}
$$ of \(A\) is an ideal of \(A^c\). Denote 
\[
F(A) := A^c / \operatorname{Ann}(A).
\]
Then  \(F(A)\) is unital for a \(\sigma\) -unital \(C^*\)-algebra \(A\) and coincides with \(A^c\) if \(A\) is unital (\cite[Proposition 1.10]{Kirch04}). Note that any 
\(^*\)-automorphism \(\alpha\) of a \(C^*\)-algebra \(A\) naturally induces a \(^*\)-automorphism of \(F(A)\). More generally, let $A$ and $B$ be \(C^*\)-algebras. 
A \(^*\)-isomorphism \(\phi : A \to B\) induces a \(^*\)-isomorphism
\[
\widetilde{\phi} : F(A) \to F(B), \quad 
\pi_{\infty}^A((a_n)) + \operatorname{Ann}(A) \mapsto \pi_{\infty}^B((\phi(a_n))) + \operatorname{Ann}(B)
\]
for \(\pi_\infty^A((a_n)) \in A^c\).

Let \( i: A \hookrightarrow B \) be an inclusion of \( C^* \)-algebras. It is worth mentioning that the induced map \( i^\infty : A^\infty \to B^\infty \) does not, in general, descend to a well-defined injective \( ^* \)-homomorphism
\[
\widetilde{i} : F(A) \to F(B), \quad \pi_\infty^A(a_n) + \operatorname{Ann}(A) \mapsto \pi_\infty^B(a_n) + \operatorname{Ann}(B),
\]
for \( \pi_\infty^A((a_n)) \in A^c \).

 Some instances where $\widetilde{i}$ is well-defined are discussed in \cite[Lemma 5.5]{AHS23}. For later use, we provide some details of a specific example here.

\begin{ex}\label{ex:sequence_algebra_bunce}
Let $B \coloneqq C \otimes D$ be any C$^*$-algebra completion of the algebraic tensor product $C \otimes_{\mathrm{alg}} D$, where $D$ is a unital C$^*$-algebra. Define $A \coloneqq C \otimes 1_D \subset B$. Then the inclusion $i : A \hookrightarrow B$ induces a well-defined \(^*\)-homomorphism $ \widetilde{i}: F(A) \to F(B)$. Indeed, let \(a\coloneqq \pi^A_{\infty}(a_1 \otimes 1, a_2 \otimes 1, \ldots) \in A^c\) 
and
\(\tilde a\coloneqq \pi^A_{\infty}(\tilde a_1 \otimes 1, \tilde a_2 \otimes 1, \ldots) \in A^c\)
such that $a+\operatorname{Ann}(A)=\tilde a+\operatorname{Ann}(A)$. 

Let $z \coloneqq \pi^B_\infty(c \otimes d, c \otimes d, \ldots) \in B$ without loss of generality. Then
$az = \pi^B_{\infty}(a_1 c \otimes d, a_2 c \otimes d, \ldots) = \pi^B_{\infty}(c a_1 \otimes d, c a_2 \otimes d, \ldots) = za
$. Hence $a \in B^c$. Similarly, $\tilde{a} \in B^c$.
Also,
$$
(a - \tilde{a}) z = \pi^B_\infty\big((a_1 - \tilde{a}_1) c \otimes d, (a_2 - \tilde{a}_2) c \otimes d, \ldots \big) = 0;
$$
similarly, $z (a - \tilde{a}) = 0$. Hence $a - \tilde{a} \in \operatorname{Ann}(B)$. Therefore, $\widetilde{i}(a) = \widetilde{i}(\tilde{a})$.

\end{ex}

The following lemma should be well known to experts; we include it for completeness.

\begin{lem}\label{lem:sequence_algebra_stablization}
Let \(A, B, C, D\) be separable \(C^*\)-algebras with inclusion maps \(i_1 : A \hookrightarrow B\) and \(i_2 : C \hookrightarrow D\). Suppose \(\phi_1 : A \to C\) and \(\phi_2 : B \to D\) are \(^*\)-homomorphisms such that
\(
i_2 \circ \phi_1 = \phi_2 \circ i_1.
\)
\begin{enumerate}
    \item If \(\phi_1\) and \(\phi_2\) are \(^*\)-isomorphisms and \(\widetilde{i}_1 : F(A) \to F(B)\) is a well-defined \(^*\)-homomorphism, then \(\widetilde{i}_2 : F(C) \to F(D)\) exists and the following diagram commutes:
    \begin{equation}\label{diag:central}
    \begin{tikzcd}
    F(A) \arrow[r, "\widetilde{i}_1"] \arrow[d, "\widetilde{\phi}_1"'] & F(B) \arrow[d, "\widetilde{\phi}_2"] \\
    F(C) \arrow[r, "\widetilde{i}_2"] & F(D)
    \end{tikzcd}
    \end{equation}

    \item If \(A \coloneqq pCp\), \(B \coloneqq pDp\), where \(p\) is a full projection in both \(C\) and \(D\), \(\phi_i\) are the canonical inclusions, and \(\widetilde{i}_1 : F(A) \to F(B)\) is a well-defined \(^*\)-homomorphism, then \(\widetilde{i}_2 : F(C) \to F(D)\) exists and fits into the commutative diagram \eqref{diag:central}.
\end{enumerate}
\end{lem}

\begin{proof}
(i)
Since $\phi_1$ is a \(^*\)-isomorphism, $\widetilde\phi_1$ is invertible. Define $\widetilde{i}_2 : F(C) \rightarrow F(D)$ by $\widetilde{i}_2 := \widetilde{\phi}_2 \circ \widetilde{i}_1 \circ \widetilde\phi_1^{-1}$. Then for $\pi_\infty^C((c_n)) + \operatorname{Ann}(C) \in F(C)$, using 
\(
i_2 \circ \phi_1 = \phi_2 \circ i_1.
\) 
we get the following 
\begin{align*}
\widetilde{i}_2(\pi_\infty^C((c_n)) + \operatorname{Ann}(C))
&= \widetilde{\phi}_2 \circ \widetilde{i}_1 \circ \widetilde\phi_1^{-1}(\pi_\infty^C((c_n)) + \operatorname{Ann}(C))\\ &= \widetilde{\phi}_2 \circ \widetilde{i}_1(\pi_\infty^A(\phi_1^{-1}(c_n)) + \operatorname{Ann}(A)) \\
&= \widetilde{\phi}_2(\pi_\infty^B(\phi_1^{-1}(c_n))  + \operatorname{Ann}(B))\\
&= \pi_\infty^D((c_n)) + \operatorname{Ann}(D).
\end{align*}

(ii) By \cite[Corollary 1.10]{Kirch04}\footnote{See Appendix (A.1) in the unpublished version of \cite{Kirch04} for details.}, the inclusion \(\phi_1\) induces a \(^*\)-isomorphism \(\widetilde{\phi}_1: F(A) \to F(C)\) defined by
\[
\widetilde{\phi}_1(c + \operatorname{Ann}(A)) = c + \operatorname{Ann}(C),
\]
with inverse \(\xi_1 \coloneqq \widetilde{\phi}_1^{-1}: F(C) \to F(A)\) given by
\[
\xi_1(c + \operatorname{Ann}(C)) = p c p + \operatorname{Ann}(A).
\]
Similarly for $\phi_2$.

Define \(\widetilde{i}_2 : F(C) \to F(D)\) by \(\widetilde{i}_2 := \widetilde{\phi}_2 \circ \widetilde{i}_1 \circ \xi_1\). Then for any \(c + \operatorname{Ann}(C) \in F(C)\), we have
\begin{align*}
\widetilde{i}_2(c + \operatorname{Ann}(C)) 
&= \widetilde{\phi}_2 \circ \widetilde{i}_1 \circ \xi_1 (c + \operatorname{Ann}(C)) = \widetilde{\phi}_2 \circ \widetilde{i}_1(p c p + \operatorname{Ann}(A)) \\
&= \widetilde{\phi}_2(p c p + \operatorname{Ann}(B)) = c + \operatorname{Ann}(D).
\end{align*}
Thus, the result follows.
\end{proof}


\subsection{Partial actions}
In this subsection, we provide a brief overview of partial actions and group actions on $C^*$-algebras. 

\begin{defn}
A \textit{partial action $\beta : G \curvearrowright A$ of a discrete group $G$ on a $C^*$-algebra $A$} is a pair consisting of a family of closed two-sided ideals $\{ I_g \}_{g \in G}$ of $A$ and a family of \(^*\)-isomorphisms $\{ \beta_g : I_{g^{-1}} \to I_g \}_{g \in G}$ satisfying the following: 
\begin{enumerate}
    \item $I_e = A$ and $\beta_e = \operatorname{id}_A$;
    \item \(\beta_h^{-1}(I_h \cap I_{g^{-1}}) \subset I_{(gh)^{-1}}\) and \(
   \beta_g \circ \beta_h(a) = \beta_{gh}(a) \quad \text{for all } a \in \beta_h^{-1}(I_h \cap I_{g^{-1}})\) and \(g,h \in G\).
\end{enumerate}
The quadruple \( (A, G, \{I_g\}_{g \in G}, \{\beta_g\}_{g \in G}) \) is called a \textit{partial dynamical system}.

\end{defn}
If \(I_g = A\) for all \(g \in G\), then \(\beta : G \curvearrowright A\) is the usual action of \(G\) on \(A\), often referred to as a \textit{global action}. Conversely, given a global action \(\alpha : G \curvearrowright A\) and a family of ideals \(\{I_g\}_{g \in G}\) of \(A\), if the maps \(\beta_g := \alpha_g|_{I_{g^{-1}}} : I_{g^{-1}} \to I_g\) are homomorphisms satisfying conditions (i)–(ii) above, then we say that \(\beta : G \curvearrowright A\) is a \textit{partial subaction} of \(\alpha\).

Now suppose \( A = C(X) \) for some compact Hausdorff space \( X \). By Gelfand duality, a partial action \( \beta : G \curvearrowright C(X) \) corresponds to a partial action \( \eta : G \curvearrowright X \) by partial homeomorphisms. Let \( I_g = C_0(X_g) \) for an open subset \( X_g \subset X \). Then each \(^*\)-isomorphism \( \beta_g : C_0(X_{g^{-1}}) \to C_0(X_g) \) corresponds to a homeomorphism \( \eta_{g^{-1}} : X_g \to X_{g^{-1}} \) such that \( \beta_g(f) = f \circ \eta_{g^{-1}} \) for all \( f \in C_0(X_{g^{-1}}) \).

Let \( \alpha : G \curvearrowright A \) be a partial action of a discrete group \( G \) on a \( C^* \)-algebra \( A \). For each \( g \in G \), denote by \( \lambda_g \) the unitary element in the multiplier algebra \( M(A \rtimes_{\alpha} G) \) implementing the action, in the sense that
\[
\alpha_g(a) = \lambda_g a \lambda_g^*
\quad \text{for all } a \in I_{g^{-1}},\, g \in G.
\]
We consider the associated partial crossed product \( A \rtimes_{\alpha} G \) as the completion of the algebraic span of \( \{ a_g \lambda_g : a_g \in I_g,\, g \in G \} \) under the full \( C^* \)-norm (\cite{Exel2017, SE17}).

The following ``associative property" of partial actions will play a role later.

\begin{prop}[{\cite[Proposition 4.2]{SE17}}]\label{prop:semidirect_partial}
Let $A$ be a $C^*$-algebra and $G \rtimes_\alpha H$ a semidirect product of groups. Suppose we are given a partial dynamical system
$(A, G \rtimes_\alpha H, \{ I_{(g,h)} \}_{(g,h) \in G \rtimes_\alpha H}, \{ \theta_{(g,h)} \}_{(g,h) \in G \rtimes_\alpha H})$
such that $I_{(g,h)} \subset I_{(g,e)} \cap I_{(e,h)}$ for every $g \in G$ and $h \in H$.

Consider the partial dynamical system
$(A, G, \{ I_{(g,e)} \}_{g \in G}, \{ \theta_{(g,e)} \}_{g \in G})$ obtained by restricting $\theta$ to the subgroup $G \times \{e\}$. Then there exists a partial action of $H$ on the crossed product $A \rtimes_{\theta} G$ defined as follows: for each $h \in H$, one has
\begin{equation}
\beta_h : I_{(e,h^{-1})} \rtimes_{\theta} G \to I_{(e,h)} \rtimes_{\theta} G, \quad \beta_h(a_g \delta_g) = \theta_{(e,h)}(a_g) \delta_{\alpha_h(g)}.
\end{equation}
Moreover, there is a $^*$-isomorphism
\begin{equation}
\varphi : A \rtimes_{\theta} (G \rtimes_\alpha H) \to (A \rtimes_\theta G) \rtimes_\beta H, \quad \varphi(a_{(g,h)} \delta_{(g,h)}) = (a_{(g,h)} \delta_g) \delta_h.    
\end{equation}
\end{prop}

\subsection{Rokhlin dimension}

Let $\operatorname{End}(A)$ denote the set of \(^*\)-endomorphisms of a \(C^*\)-algebra $A$, and $\operatorname{Aut}(A)$ denote the group of \(^*\)-automorphisms of $A$. Next, we review the definitions of the Rokhlin property and the Rokhlin dimension for actions 
of $\mathbb{Z}$.

\begin{defn}[{\cite[Definition 1.2]{BH14}}]\label{defn:Rokhlin}
Let \( A \) be a \( C^* \)-algebra and let \( \alpha \in \mathrm{Aut}(A) \). We say that \( \alpha \) has the \emph{Rokhlin property} if for every positive integer \( p \), every finite set \( F \subset A \), and every \( \varepsilon > 0 \), there exist mutually orthogonal positive contractions
\[
f_{0,0}, \dots, f_{0,p-1},\ f_{1,0}, \dots, f_{1,p} \in A
\]
such that the following hold:
\begin{enumerate}
  \item \( \left\| \left( \sum\limits_{r=0}^1 \sum\limits_{j=0}^{p-1 + r} f_{r,j} \right) a - a \right\| < \varepsilon \) for all \( a \in F \);
  \item \( \| [f_{r,j}, a] \| < \varepsilon \) for all \( r, j \) and all \( a \in F \);
  \item \( \| \alpha(f_{r,j})a - f_{r,j+1}a \| < \varepsilon \) for all \( a \in F \), \( r = 0, 1 \), and \( j = 0, \dots, p - 2 + r \);
  \item \( \| \alpha(f_{0,p-1} + f_{1,p})a - (f_{0,0} + f_{1,0})a \| < \varepsilon \) for all \( a \in F \).
\end{enumerate}
\end{defn}

Our focus now shifts to the context of endomorphisms.

\begin{defn}[{\cite[Definition 2.1]{BH14}}]\label{defn:Rokhlin_endo}
Let \( A \) be a \( C^* \)-algebra and \( \alpha \in \operatorname{End}(A)\). We say that \( \alpha \) has the \emph{Rokhlin property} if for every positive integer \( p \), every finite subset \( F \subset \alpha^p(A) \), and every \( \varepsilon > 0 \), there exist pairwise orthogonal positive contractions
\[
f_{0,0}, \dots, f_{0,p-1},\ f_{1,0}, \dots, f_{1,p} \in A
\]
satisfying conditions (i)-(iv) of Definition \ref{defn:Rokhlin}.
\end{defn}

There are different Rokhlin dimension invariants for actions of $\mathbb{Z}$. In the definition below, we use the single tower version \cite{HP15, HWZ15}.

\begin{defn}\label{defn:Rokhlin_dimension_single}
Let $A$ be a separable $C^*$-algebra, and  \( \alpha \in \mathrm{Aut}(A) \)
\begin{enumerate}
\item Let \( \alpha \in \mathrm{Aut}(A) \), 
regarded as an action of \( \mathbb{Z} \). We say that \( \alpha \) has \emph{Rokhlin dimension} \( d \) if \( d \) is the smallest non-negative integer such that for every \( \varepsilon > 0 \), every finite set \( F \subset A \), and every \( N \in \mathbb{N} \), there exist positive contractions
\[
\{ f_k^{(l)} \in A \mid 1 \leq k \leq N,\ 0 \leq l \leq d \}
\]
satisfying the following conditions for all \( a \in F \):
\begin{enumerate}
  \item[(a)] \( \left\| \left( \sum\limits_{l=0}^d \sum\limits_{k=1}^N f_k^{(l)} \right) a - a \right\| < \varepsilon \);
  \item[(b)] \( \| f_k^{(l)} f_j^{(l)} a \| < \varepsilon \) whenever \( 1\le j \neq k\le N \), for all \( 0\le l\le d \);
  \item[(c)] \( \| \alpha(f_k^{(l)}) a - f_{k+1}^{(l)} a \| < \varepsilon \) for all \( 1\le k\le N, 0\le l\le d\), with \( f_{N+1}^{(l)} := f_1^{(l)} \);
  \item[(d)] \( \| [a, f_k^{(l)}] \| < \varepsilon \) for all \( 1\le k\le N, 0\le l\le d\).
\end{enumerate}

\item Let \( G \) be a finite group, and let \( \alpha : G \to \mathrm{Aut}(A) \) be an action of \( G \) on a \( C^* \)-algebra \( A \). We say that \( \alpha \) has \emph{Rokhlin dimension} \( d \) if \( d \) is the smallest non-negative integer such 
\[
\{ f_g^{(l)} \in A \mid g \in G,\ 0 \leq l \leq d \}
\]
satisfying the following conditions for all \( a \in F \):

\begin{enumerate}
  \item[(a)] \( \left\| \left( \sum\limits_{l=0}^d \sum\limits_{g \in G} f_g^{(l)} \right) a - a \right\| < \varepsilon \);
  \item[(b)] \( \| f_g^{(l)} f_h^{(l)} a \| < \varepsilon \) for all \( g \ne h \) in \( G \), and for all \(0\le l\le d\);
  \item[(c)] \( \| \alpha_g(f_h^{(l)}) a - f_{gh}^{(l)} a \| < \varepsilon \) for all \( g, h \in G \) and all \(0\le l\le d\);
  \item[(d)] \( \| [a, f_g^{(l)}] \| < \varepsilon \) for all \( g \in G \) and all \( 0\le l \le d\).
\end{enumerate}
\end{enumerate}
\end{defn}

\begin{defn}[{\cite[Definition 2.11]{AHS23}}]\label{defn:pointwise_finite_Rokhlin}
 Let $G$ be a countable discrete group and suppose $\alpha: G \to \mathrm{Aut}(A)$ be an action. We say that $\alpha$ has \textit{pointwise finite Rokhlin dimension }if for every $s \in G \setminus \{e\}$, the restriction of $\alpha$ to the cyclic subgroup generated by $s$ has finite Rokhlin dimension.
\end{defn}

We now recall the notion of Rokhlin dimension for an action on an inclusion \( A \subset B \), as introduced in \cite{AHS23}. The essential idea is to construct Rokhlin towers that lie in \( A \), while being approximately central with respect to elements of \( B \).

\begin{defn}[{\cite[Definition 2.13]{AHS23}}]
\label{defn:Rokhlin_dim_inclusion}
Let $A \subset B$ be an inclusion of separable C$^*$-algebras.
\begin{enumerate}
    \item Suppose \( \alpha \in \mathrm{Aut}(B) \) is a \(^*\)-automorphism that leaves \( A \subset B \) invariant (viewed as an action of \( \mathbb{Z} \)). We say that \( \alpha \) has \emph{relative Rokhlin dimension \( d \)} (with respect to the inclusion \( A \subset B \)) if \( d \) is the smallest non-negative integer such that for every \( \varepsilon > 0 \), every finite set \( F \subset B \), and every \( N \in \mathbb{N} \), there exist positive contractions
\[
\{ f_k^{(l)} \in A \mid 1 \leq k \leq N,\ 0 \leq l \leq d \}
\]
satisfying the following conditions for all \( a \in F \):

\begin{enumerate}
  \item \( \left\| \left( \sum_{l=0}^d \sum_{k=1}^N f_k^{(l)} \right) a - a \right\| < \varepsilon \);
  \item \( \| f_k^{(l)} f_j^{(l)} a \| < \varepsilon \) for all \( j \ne k \) and all \( l \);
  \item \( \| \alpha(f_k^{(l)}) a - f_{k+1}^{(l)} a \| < \varepsilon \) for all \( k, l \) with \( f_{N+1}^{(l)} := f_1^{(l)} \);
  \item \( \| [a, f_k^{(l)}] \| < \varepsilon \) for all \( k, l \).
\end{enumerate}

\item Let \( G \) be a finite group, and let \( \alpha : G \to \mathrm{Aut}(B) \) be an action that leaves \( A \subset B \) invariant. We say that \( \alpha \) has \emph{relative Rokhlin dimension \( d \)} (with respect to the inclusion \( A \subset B \)) if \( d \) is the smallest non-negative integer such that for every \( \varepsilon > 0 \) and every finite subset \( F \subset B \), there exist positive contractions
\[
\{ f_g^{(l)} \in A \mid g \in G,\ 0 \leq l \leq d \}
\]
satisfying the following conditions for all \( a \in F \):

\begin{enumerate}
  \item \( \left\| \left( \sum_{l=0}^d \sum_{g \in G} f_g^{(l)} \right) a - a \right\| < \varepsilon \),
  \item \( \| f_g^{(l)} f_h^{(l)} a \| < \varepsilon \) for all \( g \ne h \in G \) and all \( l \),
\item \( \| \alpha_g(f_h^{(l)})a - f_{gh}^{(l)} a\| < \varepsilon \) for all \( g, h \in G \) and all \( l \),
 \item \( \| [a, f_g^{(l)}] \| < \varepsilon \) for all \( g \in G \) and all \( l \).
\end{enumerate}

\item Let $G$ be a countable discrete group, and suppose $\alpha: G \to \mathrm{Aut}(B)$ is an action that leaves $A \subset B$ invariant. We say that $\alpha$ has \textit{pointwise finite relative Rokhlin dimension} if for every $s \in G \setminus \{e\}$, the restriction of $\alpha$ to the subgroup generated by $s$ has finite relative Rokhlin dimension.
\end{enumerate}
\end{defn}

\begin{rmk}\label{rmk:rokhlin_relative}
Some remarks are in order.
\begin{enumerate}
\item 
Definition \ref{defn:Rokhlin} and Definition \ref{defn:Rokhlin_dimension_single} can be reformulated in terms of \( F(A) \). For example, a \(^*\)-automorphism \( \alpha \) is said to have the Rokhlin property if, for any \( m \in \mathbb{N} \), there exists a partition of unity consisting of positive contractions  
\[
f_{0,0}, f_{0,1}, \ldots, f_{0,m-1},\quad f_{1,0}, f_{1,1}, \ldots, f_{1,m} \in F(A)
\]
such that
\(
\alpha(f_{j,k}) = f_{j,k+1}\)
for $j = 0,1 \text{ and } k = 0,1,\ldots, m-2+j$,
and
\(
\alpha(f_{0,m-1} + f_{1,m}) = f_{0,0} + f_{1,0}
\) (\cite{Hirshberg22}).

\item Definition \ref{defn:Rokhlin} is stated in terms of a double Rokhlin tower. A \(^*\)-automorphism with the Rokhlin property has Rokhlin dimension at most \(1\) in the sense of our single tower definition (see \cite[Proposition 2.8]{HWZ15} for example). 
    \item Let $G$ be a countable discrete group, and let $\alpha: G \to \mathrm{Aut}(B)$ be an action that leaves the subalgebra $A \subset B$ invariant. If the \(^*\)-homomorphism $\widetilde{i}: F(A) \to F(B)$ exists and the action $\alpha|_A \in \mathrm{Aut}(A)$ has the pointwise finite Rokhlin property, then the inclusion $A \subset B$ has the pointwise finite relative Rokhlin property (\cite[Remark 2.13]{AHS23}).
\end{enumerate}
\end{rmk}

We end this subsection with the following result which exhausts all intermediate $C^*$-subalgebras for a class of inclusions. 

\begin{theo}[{\cite[Theorem 1.1]{AHS23}}]\label{thm:pointwise_finite_inclusion_criteria}
Let \( G \) be a countable discrete group with the Approximation Property (AP), and let \( B \) be a separable \( C^* \)-algebra. Let \( \alpha: G \to \mathrm{Aut}(B) \) be an action, and let \( A \subset B \) be a \( G \)-invariant subalgebra. Suppose the action of \( G \) on the inclusion \( A \subset B \) has pointwise finite relative Rokhlin dimension.

Then any intermediate \( C^* \)-algebra
\(
A \rtimes_{\alpha, r} G \subset C \subset B \rtimes_{\alpha, r} G
\)
is of the form \( B_1 \rtimes_{\alpha, r} G \), where \( A \subset B_1 \subset B \) is a \( G \)-invariant \( C^* \)-subalgebra.
\end{theo}

\section{A maximal \(C^*\)-subalgebra of $\mathcal Q_2$}\label{sec:cuntz_inclusion}

The main goal of this section is to show that the canonical copy of the Cuntz algebra \(\C\) is maximal in \(\Q\). For this, we first need to prove that \(\M\) is maximal in \(\B\) by invoking some results on partial actions. It is well known that \(\Q\) and \(\mathcal{O}_2\) can be realized as corners of crossed products \cite{BOS18, RS02}. However, some additional work is required to establish that the \( ^*\)-isomorphism from \(\Q\) to its crossed product realization restricts naturally to a \( ^*\)-isomorphism between the canonical copy of \(\mathcal{O}_2\) in \(\Q\) and its corresponding realization.

We use the same notation as in Subsection \ref{SS:Q2basics}.

First, we recall the formulation of $\mathcal{O}_2$ as a corner of a crossed product (see \cite[Section 4.2]{RS02} for more details). Let $\varphi : \F \to \F$ be an endomorphism defined by $\varphi(a) = s_2 a s_2^*$ for all $a \in \F$, $\overset{\longrightarrow}{\F}$ be the inductive limit of the sequence
$$
\F \xrightarrow{\varphi} \F \xrightarrow{\varphi} \F \xrightarrow{\varphi} \cdots,
$$
with inductive limit maps $\varphi^{(n)} : \F \to \overset{\longrightarrow}{\F}$, and $\overline{\varphi}$ be the induced \(^*\)-automorphism on $\overset{\longrightarrow}{\F}$ that satisfies 
\begin{equation} \label{eqn:inductive_limit_relationship}
\varphi^{(n)} \circ \varphi = \overline{\varphi} \circ \varphi^{(n)} 
\end{equation}
for all \(n\). The inverse of \(\overline{\varphi}\) is the map \(\varphi^{(n)}(b) \mapsto \varphi^{(n+1)}(b)\).
 

Let $w$ be the unitary in the multiplier algebra of $\overset{\longrightarrow}{\F} \rtimes_{\overline{\varphi}} \mathbb{Z}$ that implements the action $\overline{\varphi}$ of the crossed product $\overset{\longrightarrow}{\F} \rtimes_{\overline{\varphi}} \mathbb{Z}$, and $p := \varphi^{(1)}(1)$. Set \(t_2 \coloneqq wp \in p(\overset{\longrightarrow}{\F} \rtimes_{\overline{\varphi}} \mathbb{Z})p\). Then $p$ is a full projection in the crossed product $\overset{\longrightarrow}{\F} \rtimes_{\overline{\varphi}} \mathbb{Z}$, $t_2$ is an isometry in \(p(\overset{\longrightarrow}{\F} \rtimes_{\overline{\varphi}} \mathbb{Z})p\), and there is a $^*$-isomorphism \begin{equation}\label{eqn:iso_cuntz}
\phi : \mathcal{O}_2 \rightarrow p(\overset{\longrightarrow}{\F} \rtimes_{\overline{\varphi}} \mathbb{Z})p 
\end{equation}
given by \(\phi(s_2) = t_2 \) and \(\phi(a) = \varphi^{(1)}(a)\) for all \(a \in \F\). In particular, $t_1 \coloneqq \phi(s_1) = w \varphi^{(2)}(s_1 s_2^*)$.

The endomorphism \(\varphi: \F \to \F\) extends to an endomorphism \(\varphi : \BB \to \BB\), which is still denoted by \(\varphi(a)= s_2 a s_2^*\) by abusing notation. Moreover, we define \(\varphi^{(n)}, \overline{\varphi}\), and \(\overset{\longrightarrow}{\BB}\) accordingly.

\begin{prop}\label{prop:cuntz_inclusion}
Let $\overset{\longrightarrow}{\BB}$, $\overset{\longrightarrow}{\F}$, $p$, and $\overline{\varphi}$ be as above. Then 
\(
\phi : \mathcal{O}_2 \to p\big( \overset{\longrightarrow}{\F} \rtimes_{\overline{\varphi}} \mathbb{Z} \big) p
\)
as defined in \eqref{eqn:iso_cuntz} extends to a $^*$-isomorphism
\[
\phi : \Q \to p\big( \overset{\longrightarrow}{\BB} \rtimes_{\overline{\varphi}} \mathbb{Z} \big) p.
\]
\end{prop}
\begin{proof}
Note that \(\Q = C^* (\BB , s_2)\) and \(p(\overset{\longrightarrow}{\BB} \rtimes_{\overline{\varphi}} \mathbb{Z})p= C^* (\varphi^{(1)} (\BB) , t_2)\) by 
construction (see \cite[Proposition 3.3]{Stacey1993} for more details). Let \(\overline{u}= \varphi^{(1)} (u)\). Then 
\begin{align*}
    t_2 \overline{u}&= w\varphi^{(1)} (1)\varphi^{(1)} (u)= w \varphi^{(1)} (u) = \overline{\varphi}(\varphi^{(1)} (u))w \\
    &= \varphi^{(1)} (\varphi(u)) w = \varphi^{(1)} (s_2 u s_2^*) w\\
    &= \varphi^{(1)} (u^2 s_2 s_2^*)w = \varphi^{(1)} (u^2) \varphi^{(1)}(\varphi(1))w\\
   &=\varphi^{(1)} (u^2) \overline{\varphi}(\varphi^{(1)}(1))w 
   = \varphi^{(1)} (u^2) wp = \overline{u}^2 t_2.
\end{align*}
Similarly, it follows from \(p = \varphi^{(2)} (s_2 s_2^*)\), \(\varphi^{(2)} \circ \varphi = \varphi^{(1)} \), and \eqref{eqn:inductive_limit_relationship} that 
\begin{align*}
        \overline{u}t_2
        & =\overline{u}wp= w\overline{\varphi}^{-1}(\overline{u})p= w \varphi^{(2)} (u)p \\
        &= w \varphi^{(2} (u)\varphi^{(2)} (s_2 s_2^*)= w \varphi^{(2)} (s_1 s_2^*) = t_1.
\end{align*}
Hence, by the universal property of \(\Q\), we have an injective 
\(^*\)-homomorphism \(\phi : \Q \to p(\overset{\longrightarrow}{\BB} \rtimes_{\overline{\varphi}} \mathbb{Z})p\)
such that $\phi(s_2)=t_2$ and $\phi(u)=\overline{u}$.

Furthermore, \(\phi\) is surjective. To see this, let \(u^m s_{2^n} s_{2^n} u^k \in \BB\). Then we have 
\begin{align*}
\varphi^{(1)} (u^m s_{2^n} s_{2^n}^* u^k) 
&= \varphi^{(1)}(u^m)\varphi^{(1)} ( s_{2^n} s_{2^n}^*) \varphi^{(1)} (u^k) \\
 &= \phi(u^m) \phi(s_{2^n} s_{2^n}^*) \phi(u^k)\\
& = \phi\big( u^m s_{2^n} s_{2^n}^* u^k\big).
\end{align*}
Therefore $\varphi$ is a \(^*\)-isomorphism.
\end{proof}

Next, we recall \( \Q \) as a partial crossed product (see \cite[Example 4.5]{ES16} for more details) below: Let \( \langle 2 \rangle \coloneqq \{ 2^k : k \in \mathbb{Z} \} \) be the multiplicative subgroup of \( \mathbb{Q}^\times \), and let \( G \coloneqq \mathbb{Z}\left[\tfrac{1}{2}\right] \rtimes \langle 2 \rangle \) be the semidirect product with group operation
\[
(m_1, 2^{k_1})(m_2, 2^{k_2}) = (m_1 + 2^{k_1} m_2,\, 2^{k_1 + k_2})
\]
and inverse
\[
(m, 2^k)^{-1} = (-2^{-k} m,\, 2^{-k}).
\]
Let $\Delta\subset G$ be given by 
\begin{align*}
\Delta
&=\left\{(m_1, 2^{k_1})(m_2, 2^{k_2})^{-1}\in G: k_i \in \mathbb{N}, 0 \leq m_i < 2^{k_i}, i=1,2\right\}\\
&= \left\{\left(m_1 - m_2 \cdot 2^{k_1 - k_2},\, 2^{k_1 - k_2} \right)\in G: k_i \in \mathbb{N}, 0 \leq m_i < 2^{k_i}, i=1,2\right\}.
\end{align*}
Let $X \coloneqq \mathbf{Z}_2$, the $2$-adic integers, and
\begin{equation*}
X_{(m, 2^k)} \coloneqq
\begin{cases}
\emptyset & \text{if } (m, 2^k) \notin \Delta, \\
\z & \text{if } (m, 2^k) \in \Delta \text{ and } k \le 0, \\
2^k \z + m & \text{if } (m, 2^k) \in \Delta \text{ and } k > 0.
\end{cases}
\end{equation*}
For \( g := (m, 2^k) \in \Delta \), define
\begin{equation}
\eta_g : X_{g^{-1}} \to X_g,\quad \eta_g(x) = 2^k x + m.
\end{equation}
Then the family of partial homeomorphisms \( \{\eta_g\}_{g \in G} \) defines a partial action \( \eta : G \curvearrowright \z \), and there exists a \(^*\)-isomorphism
\begin{equation}\label{eqn:isomorphism_Q_partial}
\phi : \Q \to C(\z) \rtimes_{\eta} G,
\quad
\phi(s_2) = \mathbf{1}_{2 \mathbf{Z}_2} \, \lambda_{(0, 2)},
\quad
\phi(u) = \lambda_{(1, 1)}, 
\end{equation}
where \(\mathbf{1}_{m + 2^k \mathbf{Z}_2}\) is the indicator function of the set \(m + 2^k \mathbf{Z}_2\). 

\begin{lem}\label{lem:partial_car}
There is a \(^*\)-isomorphism \(\xi : \B \to C(\z) \rtimes_{\alpha} \mathbb{Z}\) which restricts to a \(^*\)-isomorphism \(\xi : \M \to C(\z) \rtimes_{\beta} \mathbb{Z}\), 
where \(\alpha : \mathbb{Z} \curvearrowright \z\) is a group action induced from $\eta$ and  \(\beta \) is a partial subaction of $\alpha$.
\end{lem}

\begin{proof}
Let \(\phi\) be the isomorphism defined in \eqref{eqn:isomorphism_Q_partial}. Then one has
\begin{equation*}
\phi(s_{2^n} s_{2^n}^*) 
= \phi(s_{2^n}) \, \phi(s_{2^n})^* = \mathbf{1}_{2^n \z} \lambda_{(0, 2^n)} \cdot \lambda_{(0, 2^{-n})} \mathbf{1}_{2^n \z} = \mathbf{1}_{2^n \z}\lambda_{(0,1)}, 
\end{equation*}
and
\begin{equation*}
\phi(u^m s_{2^n} s_{2^n}^* u^{-m}) =\lambda_{(m,1)}\mathbf{1}_{2^n \z}\lambda_{(0,1)} \lambda_{(-m,1)}= \mathbf{1}_{m+2^n \z}\lambda_{(0,1)}.
\end{equation*}
Thus
\begin{align}
\nonumber
\phi(u^m s_{2^n} s_{2^n}^* u^{k})&= \phi(u^m s_{2^n} s_{2^n}^* u^{-m})\phi(u^{k+m}) = \mathbf{1}_{m+2^n \z}\lambda_{(0,1)}\lambda_{(k+m, \, 1)}\\
\label{eqn:image_of_bunce}
&= \mathbf{1}_{m+2^n \z } \lambda_{(k+m, \, 1)}.  
\end{align}
Let \( \alpha : \mathbb{Z} \curvearrowright \z \) be the restriction of \( \eta \) to \( \mathbb{Z} \times 1 \). Hence
\[
\alpha_m:\z\to \z, \ x\mapsto x+m\quad \text{for all } x\in \z.
\]
Then clearly \( \alpha \) defines a group action. Since \(C(\z) 
= \overline{\operatorname{span}} \{\mathbf{1}_{m+2^n \z }: m\in \mathbb{Z},n\in \mathbb{N}\}\), from \eqref{eqn:defn_bunce} and \eqref{eqn:image_of_bunce} it follows that \(\phi(\BB) =C(\z)\rtimes_{\alpha}\mathbb{Z}\). 

Let
\begin{equation*}
U_{m} \coloneqq
\begin{cases}
\z & \text{if } m = 0, \\
\z \setminus \{0, 1, \ldots, m-1\} & \text{if } m > 0, \\
\z \setminus \{m, m+1, \ldots, -1\} & \text{if } m < 0.
\end{cases}
\end{equation*}
For \(m\in \mathbb{Z}\), define 
\begin{equation*}
\beta_m : U_{-m} \to U_{m}, \quad \beta_m (x)= x+m. 
\end{equation*}
Then \(\beta_m = \alpha_{m}|_{U_{-m}}\) and  $\beta : \mathbb{Z} \curvearrowright \z$ is a partial subaction of \(\alpha\).

Observe that 
\begin{align}
\label{E:Um}
U_m = \bigcup\limits_{n=1}^{\infty}\bigcup\limits_{k=0}^{2^n -1} ( k +m+ 2^n \z) =  \bigcup_{n=1}^\infty \bigcup_{\substack{m \le \ell \le 2^n -1}} (\ell + 2^n \z) 
\end{align}
for all $m\in \mathbb{Z}$.
Let \(n \in \mathbb{N}\) and suppose \(0 \leq m, k \leq 2^n - 1\). Set \(r := m - k\); then clearly \(r \leq m \leq 2^n - 1\). 
By \eqref{E:Um}, this implies \(\mathbf{1}_{m + 2^n \z} \in C_0(U_r)\). Hence, we have
\[
\phi\bigl(u^m s_{2^n} s_{2^n}^* u^{-k}\bigr) = \mathbf{1}_{m + 2^n \z} \, \lambda_{(m - k, 1)} \in C(\z) \rtimes_{\beta} \mathbb{Z}.
\]
Therefore, from \eqref{eqn:defn_car}, we have \(\phi(\F) = C(\z) \rtimes_{\beta} \mathbb{Z}\).




Let $\psi$ be a $^*$-isomorphism that identifies $\BB$ with the Bunce–Deddens algebra $M_{2^\infty}(C(\mathbb{T}))$, whose restriction to $\F$ yields an isomorphism onto the CAR algebra $\M$ (see \cite{BOS18,VY25} for example). Then \(\xi \coloneqq \phi \circ \psi^{-1}\) is the desired \(^*\)-isomorphism.
\end{proof}

One consequence of Lemma \ref{lem:partial_car} yields a key ingredient for the main result of this section.

\begin{theo}\label{thm:maximality_of_car}
\(\M\) is a maximal \(C^*\)-subalgebra of \(\B\). 
\end{theo}

\begin{proof}
By Lemma \ref{lem:partial_car}, it suffices to show that \( C(\z) \rtimes_{\beta} \mathbb{Z} \) is a maximal \(C^*\)-subalgebra of \( C(\z) \rtimes_{\alpha} \mathbb{Z} \). We keep the same notation as in that lemma.

Since \(\alpha_m(x) = x + m = x\) implies \(m = 0\), the action \(\alpha\) is free. Hence, by \cite[Corollary 5.6]{BEFRPR21} and its equivalent formulation in \cite[Theorem C]{KD24}, every intermediate \(C^*\)-subalgebra of the inclusion \( C(\z) \subset C(\z) \rtimes_{\alpha} \mathbb{Z} \) is a crossed product associated with a partial subaction of \(\alpha\).

Consequently, every intermediate \(C^*\)-subalgebra between \( C(\z) \rtimes_{\beta} \mathbb{Z} \) and \( C(\z) \rtimes_{\alpha} \mathbb{Z} \) is also associated with a partial subaction of \(\alpha\).

Let $C(\z) \rtimes_{\gamma} \mathbb{Z} = \overline{\operatorname{span}} \{ a_m \delta_m : a_m \in V_m \}$ be an intermediate $C^*$-subalgebra between $C(\z) \rtimes_{\beta} \mathbb{Z}$ and $C(\z) \rtimes_{\alpha} \mathbb{Z}$, where $\gamma_m : V_{-m} \to V_m$ defines a partial subaction of $\alpha$. Then \(U_m \subset V_m\) for each \(m\in \mathbb{Z}\). Observe that if \(V_{-1} = \z\), then \(V_{1}= \z\) and \(\gamma\) coincides with \(\alpha\). 
In what follows, our goal is to show that if there is $m\in \mathbb Z$ such that $U_m\subsetneq V_m$, then $V_{-1}=\z$.

Suppose for some $s \geq 1$, there is a $-s \leq t \leq -1$ such that $t \in V_{-s}$. Since $t \leq -1$, we have $t + 1 \leq 0$, and so $t \in U_{t+1} \subset V_{t+1}$. Hence, $t \in V_{t+1} \cap V_{-s}$. By the definition of a partial action,
\begin{equation}\label{eqn:partial_defn_1}
\gamma_{t+1}^{-1}(V_{t+1} \cap V_{-s}) \subset V_{-(t+1 + s)}. 
\end{equation}
Since $t \in V_{t+1} \cap V_{-s}$, the inclusion \eqref{eqn:partial_defn_1} implies $
-1 = \gamma_{t+1}^{-1}(t) \in V_{-(t+1 + s)}
$. Applying $\gamma_{t+1+s}$, we obtain
\begin{equation}\label{eqn:gamma_calc_1}
\gamma_{t+s+1}(-1) = t + s \in V_{t+s+1}.
\end{equation}
Since $-s \leq t$, we have $t + s \geq 0$, and therefore $t + s \in U_{t+s} \subset V_{t+s}$. Thus, we get from \eqref{eqn:gamma_calc_1} that
$$
t + s \in V_{t+s+1} \cap V_{t+s}.
$$
By the partial action definition again,
$$
\gamma_{t+1+s}^{-1}(V_{t+1+s} \cap V_{t+s}) \subset V_{-1}.
$$
So $-1 = \gamma_{t+1+s}^{-1}(t + s) \in V_{-1}$. 
Hence $\{-1\}\cup U_{-1}=\z\subset V_{-1}$, implying $V_{-1}=\z$.
Therefore, $\gamma$ is the global action $\alpha$.

Suppose for some $s \geq 1$, there exists $0 \leq t \leq s - 1$ such that $t \in V_s$. Then $t - s \in V_{-s}$. Since $-s \leq t - s \leq -1$, the previous case applies. Hence, $\gamma$ must be the global action $\alpha$. 

Therefore, there are no proper intermediate $C^*$-subalgebra between $C(\z) \rtimes_{\beta} \mathbb{Z}$ and $C(\z) \rtimes_{\alpha} \mathbb{Z}$, as desired.
\end{proof}

\begin{lem}[{\cite[Lemma 4.5]{ER24}; \cite[Lemma 5.2]{Suzuki2020}}]\label{lem:inclusion_correspondence}
Let \( A \subset B \) be a nondegenerate inclusion of \( C^* \)-algebras, and let \( p \in A \) be a full projection. Then the map \( D \mapsto pDp \) defines a bijection between the intermediate \( C^* \)-subalgebras of \( A \subset B \) and those of \( pAp \subset pBp \).

\end{lem}
\begin{cor}\label{cor:maximality_of_stablized_car}
\(\K \otimes \M\) is a maximal \(C^*\)-subalgebra of \(\K \otimes \B\). 
\end{cor}

\begin{proof}
Let \(p \coloneqq e_{11} \otimes 1\). Then \(p\) is a full projection in \(\K \otimes \M\). 
Moreover, there exists a \(^*\)-isomorphism 
\(p(\K \otimes \B)p \cong \B\), which restricts to a \(^*\)-isomorphism \(p(\K \otimes \M)p \cong \M\). The claim then follows from Lemma \ref{lem:inclusion_correspondence} and Theorem \ref{thm:maximality_of_car}.
\end{proof}

\begin{lem}\label{lem:Roklin_iso}
Let \(\phi : A \to A\) and \(\psi : B \to B\) be \(^*\)-endomorphisms. Suppose there exists a \(^*\)-isomorphism \(\theta : A \to B\) such that \(\theta \circ \phi = \psi \circ \theta\). If \(\phi\) has the Rokhlin property, then so does \(\psi\).
\end{lem}

\begin{proof}
Observe that \(\theta \circ \phi = \psi \circ \theta\) implies \(\theta \circ \phi^p = \psi^p \circ \theta\) for any positive integer \(p\). 
By the Rokhlin property of \(\phi\), there exist mutually orthogonal positive contractions
\(
f_{0,0}, \dots, f_{0,p-1},\ f_{1,0}, \dots, f_{1,p} \in A
\)
satisfying conditions (i)–(iv) of Definition \ref{defn:Rokhlin_endo} for \(\phi\). Then the elements 
\(
\theta(f_{0,0}), \dots, \theta(f_{0,p-1}),\ \theta(f_{1,0}), \dots, \theta(f_{1,p})
\)
satisfy conditions (i)–(iv) of Definition \ref{defn:Rokhlin_endo} for \(\psi\). Hence, \(\psi\) has the Rokhlin property.
\end{proof}

We are now ready to prove the main result of this section, which answers \cite[Question 9.1]{ACR21}.

\begin{theo}\label{thm:maximality_of_cuntz}
\(\mathcal{O}_2\) is a maximal \(C^*\)-subalgebra of \(\Q\).
\end{theo}
\begin{proof}
By Proposition \ref{prop:cuntz_inclusion} and Lemma \ref{lem:inclusion_correspondence}, it suffices to demonstrate the maximality of 
\(\overset{\longrightarrow}{\F} \rtimes_{\overline{\varphi}} \mathbb{Z} \subset \overset{\longrightarrow}{\BB} \rtimes_{\overline{\varphi}} \mathbb{Z}\).

Identify \(\F\) with \(\bigotimes_{n=1}^\infty M_2(\mathbb{C})\) as in \cite[Page 71]{RS02}. Then \(\varphi\) may be identified with the shift endomorphism \(\rho\) on \(\bigotimes_{n=1}^\infty M_2(\mathbb{C})\) that maps the elementary tensor \(c_1 \otimes c_2 \otimes \cdots\) to \(e_{11} \otimes c_1 \otimes c_2 \otimes \cdots\), where \(e_{11}\) denotes the standard matrix unit in \(M_2(\mathbb{C})\). Since \(\rho\) has the Rokhlin property for endomorphisms (see \cite[Definition~2.1]{BH14} and the example following it), it follows that \(\varphi\) does as well, by Lemma~\ref{lem:Roklin_iso}. Hence, by \cite[Proposition~2.2]{BH14}, the induced $^*$-automorphism \(\overline{\varphi} \in \mathrm{Aut}(\overset{\longrightarrow}{\F})\) has the Rokhlin property.

There exists a $^*$-isomorphism 
\(\psi: \overset{\longrightarrow}{\BB} \to \K \otimes \B\) 
which restricts to a $^*$-isomorphism 
\(\psi: \overset{\longrightarrow}{\F} \to \K \otimes \M\) (see, for example, \cite[Section~6]{LL12}). 
Choose a full projection \(p \in \K \otimes \M\) such that the $^*$-isomorphism 
\(p(\K \otimes \B)p \cong \B\) restricts to a $^*$-isomorphism 
\(p(\K \otimes \M)p \cong \M\).  
Then, by Example \ref{ex:sequence_algebra_bunce} and Lemma \ref{lem:sequence_algebra_stablization}, 
the map 
\[
\widetilde{i}: F(\overset{\longrightarrow}{\F}) \to F(\overset{\longrightarrow}{\BB}),
\]
induced from the inclusion 
\(\overset{\longrightarrow}{\F} \hookrightarrow \overset{\longrightarrow}{\BB}\),
is well-defined. Consequently, \(\overline{\varphi}\) has pointwise finite relative Rokhlin dimension with respect to the inclusion \(\overset{\longrightarrow}{\F} \subset \overset{\longrightarrow}{\BB}\) (see Remark \ref{rmk:rokhlin_relative}).

By the definition of \(\psi\) and Corollary \ref{cor:maximality_of_stablized_car}, \(\overset{\longrightarrow}{\F}\) is a maximal \(C^*\)-subalgebra of \(\overset{\longrightarrow}{\BB}\). Thus, there is no proper intermediate \(C^*\)-subalgebra between \(\overset{\longrightarrow}{\F} \rtimes_{\overline{\varphi}} \mathbb{Z}\) and \(\overset{\longrightarrow}{\BB} \rtimes_{\overline{\varphi}} \mathbb{Z}\) by Theorem \ref{thm:pointwise_finite_inclusion_criteria}. Hence, the claim follows.
\end{proof}
As a consequence of Theorem \ref{thm:maximality_of_cuntz}, we provide an answer to \cite[Questions 9.20 \& 9.21]{ACR21} in Corollary \ref{C:9.20} below. Before stating the consequence, recall the unitary normalizer
$$
N_{\C}(\Q) \coloneqq \{ v \in \Q : v \text{ is unitary and } v \C v^*,\ v^* \C v \subset \C \}.
$$
\begin{cor}
\label{C:9.20}
The inclusion $\mathcal O_2\subset \mathcal Q_2$ is singular:
$N_{\mathcal O_2}(\mathcal Q_2) = \mathcal U(\mathcal O_2)$, the set of all  unitaries in $\mathcal O_2$.
\end{cor}

\begin{proof}
 By \cite[Theorem 7.4]{ACR20}, the \(C^*\)-algebra generated by $N_{\C}(\Q)$
is a proper \(C^*\)-subalgebra of \(\Q\). Clearly, the unitary group \(\mathcal{U}(\C)\) is contained in \(N_{\C}(\Q)\), so we have the inclusion
\[
C^*(\mathcal{U}(\C)) \subset C^*(N_{\C}(\Q)) \subsetneq \Q,
\]
which is equivalent to
\[
\C \subset C^*(N_{\C}(\Q)) \subsetneq \Q.
\]
By Theorem \ref{thm:maximality_of_cuntz}, we get that \(\C = C^*(N_{\C}(\Q))\).
Hence $N_{\C}(\Q) \subset \mathcal U(\mathcal O_2)$, 
and the result follows immediately.
\end{proof}
\begin{rmk}
Though every copy of the Cuntz algebra can be identified with 
\(p\big((\K \otimes \M) \rtimes_{\overline{\varphi}} \mathbb{Z}\big)p\) 
for some full projection \(p \in \K \otimes \M\) \cite[Section 4.2]{RS02}, 
Proposition \ref{prop:cuntz_inclusion} plays an important role in showing 
that the canonical copy of the Cuntz algebra in \(\Q\) is maximal in \(\Q\).  
In particular, not every copy of the Cuntz algebra in \(\Q\) is maximal in \(\Q\). 
For example, let  
\[
T \coloneqq \frac{1}{\sqrt{2}}(s_1 + s_2), 
\quad 
V \coloneqq \frac{1}{\sqrt{2}}(s_1 - s_2)(s_1 s_1^* - s_2 s_2^*).
\]
Then \(C^*(T, V) \cong \C\) by \cite[Theorem 1.4]{CL12}, but 
\[
C^*(T, V) \subsetneq \C \subset \Q.
\]

\end{rmk}

\section{A symmetry of \(\mathcal{Q}_2\) and its associated maximal subalgebras}\label{sec:cuntz_inclusion_sym}

Let $A$ be a unital $C^*$-algebra and let $\sigma: A \to A$ be a \(^*\)-automorphism of order 2. The crossed product $A \rtimes_{\sigma} \mathbb{Z}_2$ is the unital $C^*$-algebra generated by $A$ and a unitary $W$ satisfying $W^2 = 1$ and the following universal property: for any unital $C^*$-algebra $B$ and unital \(^*\)-homomorphism $\psi: A \to B$, if there exists a unitary $V \in B$ such that $V^2 = 1$ and $V \psi(a) V^* = \psi(\sigma(a))$ for all $a \in A$, then there exists a unique unital \(^*\)-homomorphism $\tilde{\psi} : A \rtimes_\sigma \mathbb{Z}_2 \to B$ extending $\psi$ and satisfying $\tilde{\psi}(W) = V$.

Let $\sigma : \C \to \C$ be the flip-flop \(^*\)-automorphism of the Cuntz algebra that interchanges $s_1$ and $s_2$. This \(^*\)-automorphism extends uniquely to an outer \(^*\)-automorphism $\sigma : \Q \to \Q$ defined by $\sigma(s_2) = s_1$ and $\sigma(u) = u^*$. 
So we obtain a crossed product C*-algebra $\Q \rtimes_{\sigma} \mathbb{Z}_2$. 
Clearly $\BB$ is $\sigma$-invariant. 

It turns out that 
$\Q \rtimes_{\sigma} \mathbb{Z}_2$, $\Q$, and $\Q^\sigma$ are isomorphic to each other. 
Moreover, the maximality of the inclusion $\C \subset \Q$ also holds for $\C \rtimes_{\sigma} \mathbb{Z}_2 \subset \Q \rtimes_{\sigma} \mathbb{Z}_2$ and for $\C^\sigma \subset \Q^\sigma$.

\subsection{Two realizations of $\Q \rtimes_{\sigma} \mathbb{Z}_2$}
We will see that the crossed product $\Q \rtimes_{\sigma} \mathbb{Z}_2$
can be identified as a C*-algebra of certain endomorphism of the group 
$\mathbb{Z} \rtimes \mathbb{Z}_2$ and a corner of a crossed product. 
Both identifications will be useful in the proofs of the main results in this section. 

\subsubsection{The first realization}

To analyze the structure of $\Q \rtimes_{\sigma} \mathbb{Z}_2$, 
we will identify it as the C*-algebra associated with a certain endomorphism of the group $\mathbb{Z} \rtimes \mathbb{Z}_2$.

We first recall a definition from {\cite{Hirshberg2002}}.

\begin{defn}
\label{D:Epi}
Let $\pi: G \to G$ be an injective group homomorphism such that
\[
\bigcap_{n=0}^\infty \pi^n(G) = \{1_G\} \quad \text{and} \quad [G : \pi(G)] < \infty.
\]
The $C^*$-algebra $\mathcal{E}_\pi$ is the universal $C^*$-algebra generated by $C_r^*(G)$ and an isometry $S$ satisfying the relations:
\begin{enumerate}
    \item $S^* \lambda_x S = 0$ for all $x \notin \pi(G)$;
    \item $S \lambda_x = \lambda_{\pi(x)} S$ for all $x \in G$;
    \item If $x_1, \dots, x_n \in G$ is a complete list of right coset representatives of $\pi(G)$, then
    \[
    \sum_{k=1}^n \lambda_{x_k^{-1}} S S^* \lambda_{x_k} = 1.
    \]
\end{enumerate}
Here $\lambda_x$ denotes the canonical image of $x \in G$ in $C_r^*(G)$.
\end{defn}

To simplify our writing, from now on let
\begin{equation}\label{E:G}
    G := \mathbb{Z} \rtimes \mathbb{Z}_2,
\end{equation}
unless otherwise specified. Then $G$ is a group under the operation
$$
(n, \epsilon)(m, \delta) = \big(n + (-1)^\epsilon m,\; (\epsilon + \delta) \bmod 2\big).
$$
\begin{prop}\label{prop:crossed_alternative}
The crossed product \( \Q \rtimes_{\sigma} \mathbb{Z}_2 \) is the universal \(C^*\)-algebra \( \mathcal{E}_\pi \) associated with the group homomorphism \( \pi : G \to G \) defined by
\(
\pi(m, \epsilon) = (2m - \epsilon,\, \epsilon).
\)
\end{prop}

\begin{proof}
Notice that \( \pi \) is an injective group homomorphism satisfying  
\[
\bigcap_{n=0}^\infty \pi^n(G) = \{(0,0)\} \quad \text{and} \quad [G : \pi(G)] = 2.
\]

Then \( \mathcal{E}_\pi \) is the universal \( C^* \)-algebra generated by a unitary \( \lambda_{(1,0)} \), a unitary \( \lambda_{(0,1)} \) of order 2, and an isometry \( S \) satisfying relations \textup{(i)}–\textup{(iii)} of \( \mathcal{E}_\pi \) above. It is not hard to verify that the generators \( s_2 \), \( u \), and \( W \) of \( \Q \rtimes_{\sigma} \mathbb{Z}_2 \) satisfy these same relations.  

Next, we show that \( \Q \rtimes_{\sigma} \mathbb{Z}_2 \) is universal for the relations \textup{(i)}–\textup{(iii)} of \( \mathcal{E}_\pi \). Let \( v \) be a unitary, \( w \) a unitary of order 2, and \( s \) an isometry in some arbitrary \( C^* \)-algebra \( B \) satisfying the same relations \textup{(i)}–\textup{(iii)}.

The map \( s_2 \mapsto s \), \( u \mapsto v \) extends to a \( ^* \)-homomorphism \( \psi : \Q \to B \), since \( s \) and \( v \) satisfy relations \textup{(ii)} and \textup{(iii)}, corresponding to relations \textup{(I)} and \textup{(II)} of \( \Q \).

Moreover, the relations \textup{(i)}–\textup{(iii)} imply:
\[
w \psi(s_2) w = \psi(\sigma(s_2)), \quad \text{and} \quad w \psi(u) w = \psi(\sigma(u)).
\]
Hence, by the universal property of the crossed product \( \Q \rtimes_{\sigma} \mathbb{Z}_2 \), the map \( \psi \) extends to a \( ^* \)-homomorphism
\[
\tilde{\psi} : \Q \rtimes_{\sigma} \mathbb{Z}_2 \to B.
\]
It follows that \( \Q \rtimes_{\sigma} \mathbb{Z}_2 \) is universal for the relations \textup{(i)}–\textup{(iii)}.
\end{proof}

\begin{rmk}
\label{R:ideEpi}
Proposition \ref{prop:crossed_alternative} allows us to use the notation  
\(s_2\), \(u\), and \(W\) interchangeably with \(S\), \(\lambda_{(1,0)}\), and \(\lambda_{(0,1)}\), respectively. 
Likewise, we use \(\Q \rtimes_\sigma \mathbb{Z}_2\) interchangeably with \(\mathcal{E}_\pi\).
\end{rmk}
\subsubsection{The second realization}

In this subsection, we highlight the structure of \( \mathcal{E}_\pi \) and some of its \( C^* \)-subalgebras for the group homomorphism \( \pi \) defined in Proposition \ref{prop:crossed_alternative} (see \cite{Hirshberg2002} for the general setting). Let \( \gamma : \mathbb{T} \curvearrowright \mathcal{E}_\pi \) be the action of \( \mathbb{T} \) on \( \mathcal{E}_\pi \) given by
\[
\gamma_t(S) = tS, \quad \gamma_t(a) = a \quad \text{for all } a \in C_r^*(G),\; t \in \mathbb{T}.
\]
The fixed-point algebra \( \mathcal{F} \) of \( \gamma \) is given by
\begin{equation}
\label{eqn:guage_algebra_sym}
\mathcal{F} = \overline{\operatorname{span}} \{ \lambda_x S^n S^{*n} \lambda_y : x, y \in G,\; n \in \mathbb{N} \} = \overline{\bigcup_{n=0}^{\infty} \mathcal{F}^{(n)}},    
\end{equation}
where \(\mathcal{F}^{(n)} = \overline{\operatorname{span}} \{ \lambda_x S^n S^{*n} \lambda_y : x, y \in G \}
\),
and \( \mathcal{F}^{(n)} \subset \mathcal{F}^{(n+1)} \). Moreover, for each \(n\), we have
\begin{equation*}
 \mathcal{F}^{(n)} \cong M_{2^n}(\mathbb{C}) \otimes C_r^*(G)   
\end{equation*}
and the unital \(^*\)-homomorphism \(i_0 : \mathcal{F}^{(0)} \to \mathcal{F}^{(1)}\) defined by
\begin{equation*}
\label{eqn:connecting_map_sym}
\lambda_{(1,0)} \mapsto e_{12} \otimes 1 + e_{21} \otimes \lambda_{(1,0)}, \quad
\lambda_{(0,1)} \mapsto e_{12} \otimes \lambda_{(1,1)} + e_{21} \otimes \lambda_{(1,1)}
\end{equation*}
induces the connecting maps \(i_n \coloneqq \mathrm{id}_{2^n} \otimes i_0\) and yields the inductive system
\[
\mathcal{F} = \varinjlim (\mathcal{F}^{(n)}, i_n).
\]

\begin{rmk}
Let \(\lambda_{(m, \epsilon_1)} S^n S^{*n} \lambda_{(k, \epsilon_2)} \in \mathcal{F}\). A simple computation shows that
\begin{equation*}
\begin{aligned}
    \lambda_{(m, \epsilon_1)} S^n S^{*n} \lambda_{(k, \epsilon_2)} 
    &= \lambda_{(m - \epsilon_1, 0)} S^n S^{*n} \lambda_{(\epsilon_1 + (-1)^{\epsilon_1} k, \epsilon_2)} \\
    &= \lambda_{(r, 0)} S^n S^{*n} \lambda_{(-r + m + (-1)^{\epsilon_1} k, \epsilon_2)},
\end{aligned}
\end{equation*}
where \(m - \epsilon_1 \equiv r \pmod{2^n}\). Note that \(0 \le r \le 2^n - 1\), and \(\lambda_{(-r + m + (-1)^{\epsilon_1} k, \epsilon_2)}\) can be an arbitrary in \( C_r^*(G)\). Therefore, \(\mathcal{F}\) can be reformulated as
\begin{equation*}
\label{eqn:bunce_sym_reformulation}
    \mathcal{F} = \overline{\operatorname{span}} \left\{ \lambda_{(m, 0)} S^n S^{*n} \lambda_{(k, \epsilon)} : 0 \le m \le 2^n - 1,\; (k, \epsilon) \in G,\; n \in \mathbb{N} \right\}.
\end{equation*}
\end{rmk}

Now, we state the crossed product representation of \(\mathcal{E}_\pi\). Define 
\begin{align}
\label{E:varphiF}
\varphi : \mathcal{F} \to \mathcal{F}, \ \varphi(a) = S a S^*
\end{align}
and consider the inductive limit \(\overset{\longrightarrow}{\mathcal{F}} = \varinjlim (\mathcal{F}, \varphi)\) with inductive limit maps \(\varphi^{(n)} : \mathcal{F} \to \overset{\longrightarrow}{\mathcal{F}}\) and induced \(^*\)-automorphism \(\overline{\varphi}\). Let \(w\) implement \(\overline{\varphi}\) in \(\overset{\longrightarrow}{\mathcal{F}} \rtimes_{\overline{\varphi}} \mathbb{Z}\), and set \(p := \varphi^{(1)}(1)\), \(T := w p\). Then \(p\) is a full projection in \(\overset{\longrightarrow}{\mathcal{F}}\), \(T\) is an isometry in the corner \(p(\overset{\longrightarrow}{\mathcal{F}} \rtimes_{\overline{\varphi}} \mathbb{Z}) p\), and there exists a \(^*\)-isomorphism
\begin{equation}\label{eqn:crossed_identified_sym}
\phi : \mathcal{E}_\pi \to p(\overset{\longrightarrow}{\mathcal{F}} \rtimes_{\overline{\varphi}} \mathbb{Z}) p    
\end{equation}
such that \(\phi(S) = T\) and \(\phi(a) = \varphi^{(1)}(a)\) for all \(a \in \mathcal{F}\).

\subsection{A symmetry of $\Q$}

This subsection focuses on the symmetry $\sigma$ of $\Q$ that extends the flip-flop $^*$-automorphism of $\C$. We first provide some concrete descriptions of the fixed-point algebra $\Q^\sigma$ of $\Q$ under $\sigma$. We then show that the $C^*$-algebras $\Q \rtimes_{\sigma} \mathbb{Z}_2$, $\Q$, and $\Q^\sigma$ are all isomorphic. 



The following result generalizes Theorem 1.4 of Choi and Latrémolière \cite{CL12} to $\Q$, and it also answers some questions posed in \cite{ACR21}.

\begin{theo}\label{thm:fixed_sym}
Let
\[
T \coloneqq \frac{1}{\sqrt{2}}(s_1 + s_2), \quad V \coloneqq \frac{1}{\sqrt{2}}(s_1 - s_2)(s_1 s_1^* - s_2 s_2^*), \quad R\coloneqq u+u^*.
\]
Then
\[
\Q^\sigma = C^*(T, V, R) = C^*\big(x u^h + \sigma(x) u^{-h} \mid x \in \C,\, h \in \mathbb{Z} \big).
\]
\end{theo}

\begin{proof}
Note that $C^*(T, V) = \Q^\sigma \cap \C=\mathcal{O}_2^{\sigma}$ (see \cite[Theorem 1.4]{CL12}). Since the fixed points of $\Q$ under $\sigma$ are of the form $a + \sigma(a)$ \cite[Lemma 2.1]{CL09}, we obtain
\begin{align}
\label{E:Q13}
C^*(T, V, R) \subset C^*\big(x u^h + \sigma(x) u^{-h} \mid x \in \C,\, h \in \mathbb{Z} \big) \subset \Q^\sigma. 
\end{align}
Let \( C \coloneqq C^*(T, V, R) \). Using \cite[Lemma 2.1]{CL09} 
and Remark \ref{rmk:form_adic}, it suffices to show the following: if 
\begin{align}
\label{eqn:element_cuntz}
 a \coloneqq x u^m \in \Q\ \text{ with }\
x=u^{k_1} S_{2^{n_1}} S_{2^{n_2}}^* u^{-k_2} \in \C,\quad m \ge 0,  
\end{align}
and \( a + \sigma(a) \in \Q^\sigma \), then \( a + \sigma(a) \in C \). We prove the claim by induction on \(m\).

Clearly, if \( m = 0 \), then \( a + \sigma(a) \in \C^\sigma = C^*(T, V) \subset C \). 
Now suppose \( x = u^{k_1} S_{2^{n_1}} S_{2^{n_2}}^* u^{-k_2} \in \C \) and \( m = 1 \). Then one has either \( xu \in \C \) or \( xu^* \in \C \). Indeed, if \( k_2 \ge 1 \), then \( 0 \le k_2 - 1 \le 2^{n_2} - 1 \), and thus
\[
xu = u^{k_1} S_{2^{n_1}} S_{2^{n_2}}^* u^{-(k_2 - 1)} \in \C.
\]
On the other hand, if \( k_2 = 0 \), then
\[
xu^* = u^{k_1} S_{2^{n_1}} S_{2^{n_2}}^* u^{-1} \in \C.
\]
Consequently, if \( xu \in \C \), then \( a + \sigma(a) = xu + \sigma(xu) \in \C^\sigma \subset C \). If \( xu^* \in \C \), we observe that
\begin{align*}\label{eqn:induction1}
a + \sigma(a) &= xu + \sigma(xu) \nonumber \\
&= (x + \sigma(x))(u + u^*) - (xu^* + \sigma(xu^*)) \in C
\end{align*}
as $x + \sigma(x)\in \mathcal{O}_2^\sigma$, $u+u^*\in \mathcal{Q}_2^\sigma$, and 
$xu^* + \sigma(xu^*)\in \mathcal{O}_2^\sigma$.

Fix $k\in\mathbb N$, $k\ge1$, and let $x$ be as in \eqref{eqn:element_cuntz} (so $x\in\C$ is fixed). For $m\ge1$ set $b_m:=xu^m$. 
Assume as induction hypothesis that
\[
b_m+\sigma(b_m)\in C \quad\text{for all }1\le m\le k.
\]
We will show that $b_{k+1}+\sigma(b_{k+1})\in C$.
Note that
\[
xu^{k+1} + \sigma(x u^{k+1}) = (xu^k + \sigma(x u^k))(u + u^{*}) - (x u^{k-1} + \sigma(x u^{k-1})).
\]
By the induction hypothesis, each term on the right belongs to \( C \). Hence \( xu^{k+1} + \sigma(x u^{k+1}) \in C \). 

Therefore, all three \(C^*\)-algebras in \eqref{E:Q13} coincide. 
\end{proof}

The following result shows that \( \Q \rtimes_{\sigma} \mathbb{Z}_2 \) is a Kirchberg algebra satisfying the UCT. 

\begin{prop}\label{prop:ktheory_sym}
\( \Q \rtimes_{\sigma} \mathbb{Z}_2 \) is a simple, nuclear, purely infinite \(C^*\)-algebra satisfying the UCT with
\[
K_0(\Q \rtimes_{\sigma} \mathbb{Z}_2) \cong \mathbb{Z}, \quad K_1(\Q \rtimes_{\sigma} \mathbb{Z}_2) \cong \mathbb{Z}.
\]  
\end{prop}

\begin{proof}
The outerness of $\sigma$ (\cite{ACR21}) and the amenability of \(\mathbb{Z}_2\) ensure that $\Q \rtimes_{\sigma} \mathbb{Z}_2$ is a simple, nuclear, purely infinite \(C^*\)-algebra (see \cite{Jeong95, Kishimoto81}). This crossed product satisfies the UCT since it is Morita equivalent to $\overset{\longrightarrow}{\mathcal{F}} \rtimes_{\overline{\varphi}} \mathbb{Z}$, as shown in \eqref{eqn:crossed_identified_sym}, and the latter satisfies the UCT \cite{RS02}.


It remains to compute the K-theory of \( \Q \rtimes_{\sigma} \mathbb{Z}_2 \). For this, 
let \(P_{(0,1)} = \frac{1 + \lambda_{(0,1)}}{2}\) and \(P_{(1,1)} = \frac{1 + \lambda_{(1,1)}}{2}\). It is known that \([1]\), \([P_{(0,1)}]\), and \([P_{(1,1)}]\) generate \(K_0(C_r^*(G)) \cong \mathbb{Z}^3\), and that \(K_1(C_r^*(G)) = 0\) (see \cite[6.10.4]{Bla98}). Using \eqref{eqn:connecting_map_sym}, we have
\[
\begin{aligned}
i_0(1) &= e_{11} \otimes 1 + e_{22} \otimes 1, \\
i_0(P_{(1,1)}) &= e_{11} \otimes P_{(1,1)} + e_{22} \otimes \frac{1}{2}(1+\lambda_{(2,1)}),\\
i_0(P_{(0,1)}) &= \tfrac{1}{2} e_{11} \otimes 1 + \tfrac{1}{2} e_{12} \otimes \lambda_{(1,1)} + \tfrac{1}{2} e_{21} \otimes \lambda_{(1,1)} + \tfrac{1}{2} e_{22} \otimes 1.
\end{aligned}
\]
Hence, the induced map on \(K_0\) is given by
\[
[i_0(1)] = 2 \cdot [1], \quad [i_0(P_{(1,1)})] = 2 \cdot [P_{(1,1)}], \quad [i_0(P_{(0,1)})] = [1].
\]
In matrix form, this is given by
\[
A = 
\begin{pmatrix}
2 & 0 & 1 \\  
0 & 2 & 0 \\  
0 & 0 & 0 
\end{pmatrix}.
\]
Then \(K_0(\mathcal{F})\) is the inductive limit of the system
\[
\mathbb{Z}^3 \xrightarrow{A} \mathbb{Z}^3 \xrightarrow{A} \mathbb{Z}^3 \to \cdots,
\]
where the connecting maps are given by this matrix \(A\). So, \(K_0 (\mathcal{F})  \cong \mathbb{Z}\left[\tfrac{1}{2}\right] \oplus \mathbb{Z}\) and \(K_1 (\mathcal{F})=0\).

The \(^*\)-isomorphism \(\overline{\varphi}\) on $\overset{\longrightarrow}{\mathcal{F}}$ induces an action on \(K_0(\overset{\longrightarrow}{\mathcal{F}})\) that scales the first component by \(\frac{1}{2}\) and leaves the second component unchanged. Using the Pimsner–Voiculescu exact sequence yields
\[
K_0(\overset{\longrightarrow}{\mathcal{F}} \rtimes_{\overline{\varphi}} \mathbb{Z}) = K_1(\overset{\longrightarrow}{\mathcal{F}} \rtimes_{\overline{\varphi}} \mathbb{Z}) = \mathbb{Z}.
\]
From \eqref{eqn:crossed_identified_sym} we get that
\[
K_0(\Q \rtimes_{\sigma} \mathbb{Z}_2) = K_1(\Q \rtimes_{\sigma} \mathbb{Z}_2) = \mathbb{Z}.
\]
We're done.
\end{proof}

A quick consequence of Proposition \ref{prop:ktheory_sym} is the following:

\begin{theo}
\label{T:3iso}
$\Q \rtimes_{\sigma} \mathbb{Z}_2\cong \Q\cong \Q^\sigma.$
\end{theo}

\begin{proof}
It follows from \cite[Proposition 3.4]{Rieffel80} and the simplicity of the crossed product that $\Q^\sigma$ is Morita equivalent to $\Q \rtimes_\sigma \mathbb{Z}_2$. So Proposition \ref{prop:ktheory_sym} shows that $\Q^\sigma$, $\Q \rtimes_\sigma \mathbb{Z}_2$, and $\Q$ have the same $K$-theory. Since all are unital Kirchberg algebras satisfying the UCT, the Kirchberg–Phillips classification theorem implies  
$\Q \rtimes_{\sigma} \mathbb{Z}_2\cong \Q\cong \Q^\sigma$. 
\end{proof}




\subsection{Maximal $C^*$-subalgebras related to $\mathcal{O}_2 \subset \Q$}

From Theorem \ref{thm:maximality_of_cuntz} and \ref{T:3iso}, one naturally wonders if \(\C \rtimes_{\sigma} \mathbb{Z}_2\) is maximal in \(\Q \rtimes_{\sigma} \mathbb{Z}_2\), and asks the same question about \(\C^\sigma\) in \(\Q^\sigma\). The main goal of this subsection is to show that this is indeed the case. 



\begin{prop}\label{prop:composition}
Suppose \( \phi_1, \phi_2 \in \operatorname{End}(\Q) \) with \( \phi_2 \) a \(^*\)-automorphism and  \( \phi_1|_{\BB} = \phi_2|_{\BB} \). Then there exists \( t \in \mathbb{T} \) such that \( \phi_1 = \phi_2 \circ \alpha_t\), where
$\alpha$ is the gauge action of $\mathbb{T}$ on $\Q$.
\end{prop}

\begin{proof}
Notice that \( \phi_2^{-1} \circ \phi_1(x) = x \) for all \( x \in \BB \). Using that \( \phi_2^{-1} \circ \phi_1(u) = u \) and applying \cite[Theorem 6.14]{ACR18}, we conclude that \( \phi_2^{-1} \circ \phi_1 \) is an automorphism of \(\Q\) that fixes \(\BB\). Consequently, \( \phi_2^{-1} \circ \phi_1 \) is a gauge automorphism \( \alpha_t \) for some \( t \in \mathbb{T} \), by \cite[Proposition 3.4]{ACR18b}. Therefore, \( \phi_1 = \phi_2 \circ \alpha_t \).
\end{proof}
\begin{theo}\label{thm:max_bunce_sym}
Let $\overline{\sigma}$ be the \(^*\)-automorphism of \(\B\) induced by restricting $\sigma$ to $\BB$ and identifying $\BB$ with $\B$. Then $\M \rtimes_{\overline{\sigma}} \mathbb{Z}_2$ is a maximal $C^*$-subalgebra of $\B \rtimes_{\overline{\sigma}} \mathbb{Z}_2$.
\end{theo}

\begin{proof}
First, we claim that \(\sigma|_{\BB} \) is an outer automorphism. Suppose otherwise. Then there exists a unitary \( v \in \BB \) such that the restriction of \( \phi_1 \coloneqq \operatorname{Ad}(v) : \Q \to \Q \) to \( \BB \) agrees with \( \sigma|_{\BB} \). By Proposition \ref{prop:composition}, there exists \( t \in \mathbb{T} \) such that \( \phi_1 = \sigma \circ \alpha_t \). Note that
\[
\phi_1(u) = \sigma \circ \alpha_t(u)  = \sigma(u) = u^*.
\]
By \cite[Theorem 5.9]{ACR18}, \( \phi_1 \) is an outer automorphism, a contradiction. Hence, the claim follows.

Consequently, \(\overline{\sigma}\) is an outer \(^*\)-automorphism since \(\sigma|_{\BB}\) is. It now follows from \cite[Theorem 2.3]{BEMSW15} that \(\overline{\sigma}\) has Rokhlin dimension at most \(1\). Using Example \ref{ex:sequence_algebra_bunce} and Remark \ref{rmk:rokhlin_relative}(ii), we conclude that \(\overline{\sigma}\) has pointwise finite relative Rokhlin dimension. Combining Theorem \ref{thm:maximality_of_car} and Theorem \ref{thm:pointwise_finite_inclusion_criteria} yields the desired conclusion.
\end{proof}

Next, we realize $\mathcal{F}$ (defined in \eqref{eqn:guage_algebra_sym}) as a crossed product $C^*$-algebra (see \cite[Lemma 2.6]{Hirshberg2002} for a more general setting) and single out an important $C^*$-subalgebra of $\mathcal{F}$.

\begin{theo}\label{thm:bunce_sym}
Let 
\begin{equation}
\mathcal{H} \coloneqq \overline{\operatorname{span}} \left\{ \lambda_{(m_1, 0)} S^n S^{*n} \lambda_{(-m_2, k_2)} : k_2\in \mathbb{Z}_2,\ 0 \le m_i < 2^n,\ n \in \mathbb{N} \right\}.
\end{equation}
Then \(\mathcal{H}\) is a maximal \(C^*\)-subalgebra of \(\mathcal{F}\) and there is a \(^*\)-isomorphism \(\xi : \mathcal{F} \to \B \rtimes_{\overline{\sigma}} \mathbb{Z}_2\) such that the following diagram
\begin{equation}\label{diag:bunce_symm}
\begin{tikzcd}
\mathcal{H} \arrow[r, hook] \arrow[d, "{\Phi}"', "\cong" right] & \mathcal{F} \arrow[d, "{\Phi}", "\cong" left] \\
\M \rtimes_{\overline{\sigma}} \mathbb{Z}_2 \arrow[r, hook] & \B \rtimes_{\overline{\sigma}} \mathbb{Z}_2
\end{tikzcd}
\end{equation}
commutes,
where \(\overline{\sigma} \in \operatorname{Aut}(\B)\) is the \(^*\)-automorphism induced by restricting \(\sigma\) to \(\BB\).
\end{theo}

\begin{proof}
Recall from that \eqref{E:G} that \(G=\mathbb{Z} \rtimes \mathbb{Z}_2 \). Define an action
$$
\alpha : G\curvearrowright \z
$$
by
\begin{equation}
 \alpha_{(n, \epsilon)}(x) \coloneqq 
\begin{cases} 
x + n & \text{if } \epsilon = 0, \\
- (x + 1) + n & \text{if } \epsilon = 1.
\end{cases}
\end{equation}
for all $(n, \epsilon) \in G$ and $x \in \z$. It is easy to check that this action is minimal and free. Hence, \(C(\z) \rtimes_{\alpha} G\) is simple.

The map \( \Psi : C(\z) \to \mathcal{F} \) defined by
\[
\Psi\left( \mathbf{1}_{m + 2^n \z} \right) = \lambda_{(m,0)} S^n S^{*n} \lambda_{(-m,0)}
\]
is a \(^*\)-homomorphism. It is not hard to check that 
\[
\lambda_{(1,0)} \Psi\left( \mathbf{1}_{m + 2^n \z} \right) \lambda_{(-1,0)} = \Psi\left( \mathbf{1}_{(m+1) + 2^n \z} \right)
\]
and
\[
\lambda_{(0,1)} \Psi\left( \mathbf{1}_{m + 2^n \z} \right) \lambda_{(0,1)} = \Psi\left( \mathbf{1}_{-{(m+1)} + 2^n \z} \right).
\]
By the universal property of the crossed product \(C(\z) \rtimes_{\alpha} G\), \(\Psi\) extends to an injective \(^*\)-homomorphism
\begin{equation}\label{eqn:iso_bunce_symm}
\Psi : C(\z) \rtimes_{\alpha} G \to \mathcal{F}.    
\end{equation}
Also, $\Psi$ is surjective since for any $\lambda_{(m_1,0)} S^n S^{*n} \lambda_{(m_2, \epsilon_2)} \in \mathcal{F}$, we have
$$
\Psi\big( \xi_{(m_1, 0)} \mathbf{1}_{2^n \z} \xi_{(m_2, \epsilon_2)} \big) = \lambda_{(m_1, 0)} S^n S^{*n} \lambda_{(m_2, \epsilon_2)},
$$
where $\xi$ is the left regular representation of $G$ associated to the crossed product $C(\z)\rtimes_{\alpha} G$. 
Therefore $\Psi$ is a \(^*\)-isomorphism.

Let
\begin{equation*}
U_{(n,\epsilon)} \coloneqq
\begin{cases}
\z & \text{if } n = 0, \\
\z \setminus \{0, 1, \ldots, n-1\} & \text{if } n > 0, \\
\z \setminus \{n, n+1, \ldots, -1\} & \text{if } n < 0.
\end{cases}
\end{equation*}
Define \(\beta_{(n,\epsilon)} \coloneqq \alpha_{(n,\epsilon)}|_{U_{(-n,\epsilon)}}\). Then \(\beta : G \curvearrowright \z\) is a partial subaction of \(\alpha\). 
Note that
\begin{equation*}
 U_{(n,\epsilon)} = \bigcup_{j=1}^{\infty} \bigcup_{k=0}^{2^j -1} (k + n + 2^j \z) = \bigcup_{j=1}^\infty \bigcup_{\substack{n \le \ell \le 2^j -1}} (\ell + 2^j \z)  
\end{equation*}
and 
\[
C(\z) \rtimes_{\beta} G = \overline{\operatorname{span}}\{f_g \xi_g : g \in G,\ f_g \in C_0(U_g)\}.
\]
For \(n \le \ell \le 2^j - 1\),
\[
\Psi(\mathbf{1}_{(\ell + 2^j \z)} \xi_{(n,\epsilon)}) = \lambda_{(\ell,0)} S^j S^{*j} \lambda_{(-\ell,0)} \lambda_{(n,\epsilon)} = \lambda_{(\ell,0)} S^j S^{*j} \lambda_{(-\ell +n,\epsilon)} \in \mathcal{H}.
\]
Thus \(\Psi ( C(\z) \rtimes_{\beta} G) \subset \mathcal{H}\). If \(\lambda_{( m_1, 0)} S^j S^{*j} \lambda_{(-m_2, \epsilon_2)} \in \mathcal{H}\) with \(0 \le m_i \le 2^j - 1\), then
\begin{align*}
\lambda_{( m_1, 0)} S^j S^{*j} \lambda_{(-m_2, \epsilon_2)}
&=  \left( \lambda_{(m_1, 0)} S^j S^{*j} \lambda_{(-m_1, 0)} \right) \lambda_{(m_1 - m_2, \epsilon_2)} \\
&= \Psi\left(\mathbf{1}_{m_1 + 2^j \z} \xi_{(m_1 - m_2, \epsilon_2)} \right).
\end{align*}
Since \(m_1 - m_2 \le m_1\),  
\[
\mathbf{1}_{m_1 + 2^j \z} \xi_{( m_1 - m_2, \epsilon_2)} \in C(\z) \rtimes_{\beta} G.
\]
Hence 
$
\lambda_{(m_1, 0)} S^j S^{*j} \lambda_{(-m_2, \epsilon_2)} \in \Psi(C(\z) \rtimes_{\beta} G).
$
Therefore
\begin{equation}\label{eqn:restriction_bunce_sym}
\Psi(C(\z) \rtimes_{\beta} G) = \mathcal{H}.
\end{equation}
Using Proposition \ref{prop:semidirect_partial},  we get a \(^*\)-isomorphism
\[
C(\z) \rtimes_{\alpha} G \cong (C(\z) \rtimes_{\alpha} \mathbb{Z}) \rtimes_{\gamma} \mathbb{Z}_2 \] which restricts to a \(^*\)-isomorphism \[ \quad C(\z) \rtimes_{\beta} G \cong (C(\z) \rtimes_{\beta} \mathbb{Z}) \rtimes_{\gamma} \mathbb{Z}_2,
\]
where \(\gamma : C(\z) \rtimes_{\alpha} \mathbb{Z} \to C(\z) \rtimes_{\alpha} \mathbb{Z}\), 
\(\mathbf{1}_{m + 2^s \z} \xi_{(n,0)}\mapsto \mathbf{1}_{-(m + 1) + 2^s \z} \xi_{(-n,0)}\), is a \(^*\)-automorphism of order \(2\). 

Identify \(C(\z) \rtimes_{\alpha} \mathbb{Z}\) and \(C(\z) \rtimes_{\beta} \mathbb{Z}\) with \(\B\) and \(\M\), respectively, as in Lemma \ref{lem:partial_car}. By the definition of \(\gamma\), it corresponds to a \(^*\)-automorphism \(\overline{\sigma}\) of \(\B\), induced by restricting \(\sigma\) to \(\BB\) and identifying \(\BB\) with \(\B\). Using the identifications above, together with \eqref{eqn:iso_bunce_symm} and \eqref{eqn:restriction_bunce_sym}, we obtain a \(^*\)-isomorphism 
$\Phi: \mathcal F\to \B \rtimes_{\overline{\sigma}} \mathbb{Z}_2$ such that 
that the commutative diagram in \eqref{diag:bunce_symm} holds. The maximality of the inclusion \(\mathcal{H} \subset \mathcal{F}\) then follows from Theorem \ref{thm:max_bunce_sym}.
\end{proof}

Using the universal property of $\C \rtimes_{\sigma} \mathbb{Z}_2$ and the property of \(W\), it is easy to see that the $C^*$-subalgebra $C^*(s_1, s_2, W)$ of $\Q \rtimes_{\sigma} \mathbb{Z}_2$ is $^*$-isomorphic to $\C \rtimes_{\sigma} \mathbb{Z}_2$. We refer to it as the canonical copy of $\C \rtimes_{\sigma} \mathbb{Z}_2$ inside $\Q \rtimes_{\sigma} \mathbb{Z}_2$, and still denote it by $\C \rtimes_{\sigma} \mathbb{Z}_2$. Now, we realize both \(\C \rtimes_{\sigma}\mathbb{Z}_2\) and \(\Q\rtimes_{\sigma}\mathbb{Z}_2\) as corners of crossed product \(C^*\)-algebras.

\begin{prop}\label{prop:cuntz_inclusion_sym}
Let \(\phi: \Q \rtimes_{\sigma} \mathbb{Z}_2 
\to p\big(\overset{\longrightarrow}{\mathcal{F}} \rtimes_{\overline{\varphi}} \mathbb{Z}\big)p\) 
be the \(^*\)-isomorphism defined in \eqref{eqn:crossed_identified_sym}. 
Then the restriction of \(\phi\) to \(\C \rtimes_{\sigma} \mathbb{Z}_2\) 
is a \(^*\)-isomorphism onto 
\(p\big(\overset{\longrightarrow}{\mathcal{H}} \rtimes_{\overline{\varphi}} \mathbb{Z}\big)p\).
\end{prop}
\begin{proof}
Recall from \eqref{E:varphiF} that $\varphi : \mathcal{F} \to \mathcal{F}$ is a $^*$-endomorphism defined by $\varphi(a) = SaS^*$. Notice that $\varphi(\mathcal{H}) \subset \mathcal{H}$, and thus we have a $^*$-endomorphism $\varphi : \mathcal{H} \to \mathcal{H}$. Similarly, we have connecting maps $\varphi^{(n)} : \mathcal{H} \to \overset{\longrightarrow}{\mathcal{H}}$ and an induced \(^*\)-automorphism $\overline{\varphi} : \overset{\longrightarrow}{\mathcal{H}} \to \overset{\longrightarrow}{\mathcal{H}}$. Thus,
$$
p(\overset{\longrightarrow}{\mathcal{H}} \rtimes_{\overline{\varphi}} \mathbb{Z})p \subset p(\overset{\longrightarrow}{\mathcal{F}} \rtimes_{\overline{\varphi}} \mathbb{Z})p.
$$
Using \eqref{eqn:crossed_identified_sym}, we see that $\phi(\lambda_{(1,0)}S), \phi(S), \phi(\lambda_{(0,1)}) \in p(\overset{\longrightarrow}{\mathcal{H}} \rtimes_{\overline{\varphi}} \mathbb{Z})p$. Moreover, applying a similar argument as in Proposition \ref{prop:cuntz_inclusion}, along with the construction of \(p\left(\overset{\longrightarrow}{\mathcal{H}} \rtimes_{\overline{\varphi}} \mathbb{Z}\right)p\), we obtain
$\phi(\mathcal{O}_2 \rtimes_{\sigma} \mathbb{Z}_2) = p\left(\overset{\longrightarrow}{\mathcal{H}} \rtimes_{\overline{\varphi}} \mathbb{Z}\right)p.$
\end{proof}

As at the beginning of the proof of Proposition \ref{prop:cuntz_inclusion_sym}, \(\mathcal{H}\) is invariant under the \(^*\)-endomorphism \(\varphi\) in \eqref{E:varphiF}, similarly to \(\mathcal{F}_2\).  
In what follows, we will show that there is a mapping that intertwines these two restrictions.

\begin{prop}\label{prop:rokhlin_property_car_sym}
There exists a $^*$-isomorphism $\Theta: \mathcal{F}_2 \to \mathcal{H}$ such that $\Theta \circ \varphi|_{\mathcal{F}_2} = \varphi|_{\mathcal{H}} \circ \Theta$. Consequently, the $^*$-endomorphism $\varphi|_{\mathcal H}$ satisfies the Rokhlin property.
\end{prop}

\begin{proof}

One can see that the map \(\Theta : \C \to \C \rtimes_{\sigma} \mathbb{Z}_2\), defined by
\begin{equation*}
\Theta(S) = S, \quad \Theta(\lambda_{(1,0)} S) = \lambda_{(0,1)} S
\end{equation*}
is a \(^*\)-isomorphism (see \cite[Theorem 2.1]{CL09}).  
Let \(\chi : \C \to \C\) be the \(^*\)-endomorphism defined by
\begin{equation*}
\chi(x) = S x S^* + \lambda_{(1,0)} S x S^* \lambda_{(-1,0)},
\end{equation*}
and let \(B \coloneqq \overline{\operatorname{span}}\{\lambda_{(m,0)} S S^* \lambda_{(-k,0)} : 0 \le m,k \le 1\}\).  
Then 
\[
\mathcal{F}_2 = C^*\big(\{\chi^n(B) : 0\le n \in \mathbb{Z}\}\big)
\]
(see \cite[Page 77]{RS02}). 

Let \(\overline{\chi}: \C \rtimes_{\sigma} \mathbb{Z}_2 \to \C \rtimes_{\sigma} \mathbb{Z}_2\) be the \(^*\)-endomorphism defined by \(\overline{\chi} \coloneqq \Theta \circ \chi \circ \Theta^{-1}\).  
A simple computation shows that \(\overline{\chi}(\mathcal{H}) \subset \mathcal{H}\),  
\begin{equation*}\label{eqn:induction_car_sym}
\Theta(B) = \overline{\operatorname{span}} \{S S^*,\ \lambda_{(1,0)} S S^* \lambda_{(-1,0)},\ S S^* \lambda_{(0,1)},\ \lambda_{(1,0)} S S^* \lambda_{(-1,1)}\} \subset \mathcal{H},
\end{equation*}
and
\[
\Theta(\chi^n(B)) = \overline{\chi}^n(\Theta(B)).
\]
Thus, \(\Theta(\mathcal{F}_2) \subset \mathcal{H}\).

Recall that \(\mathcal{H}\) is the closed linear span of elements of the form \(\lambda_{(m,0)} S^n S^{*n} \lambda_{(-k,\epsilon)}\), where \(0 \le m, k \le 2^n - 1\), \(\epsilon \in \mathbb{Z}_2\), and \(n\in \mathbb{N}\). We prove 
the reverse inclusion 
\(\mathcal{H}\subset \Theta(\mathcal{F}_2)\) by induction on \(n\). 

When $n=1$, 
notice that 
\begin{equation}
\label{eqn:elements_B}
\begin{split}
\lambda_{(1,0)}SS^* &= \Theta(\lambda_{(3,0)}S^2 S^{*2} + \lambda_{(1,0)}S^2 S^{*2}\lambda_{(-2,0)}), \\
\lambda_{(1,0)}SS^*\lambda_{(0,1)} &= \Theta(\lambda_{(3,0)}S^2 S^{*2}\lambda_{(-1,0)} + \lambda_{(1,0)}S^2 S^{*2}\lambda_{(-3,0)}), \\
SS^{*}\lambda_{(-1,0)} &= \Theta(S^2 S^{*2}\lambda_{(-3,0)} + \lambda_{(2,0)}S^2 S^{*2}\lambda_{(-1,0)}), \\
SS^{*}\lambda_{(-1,1)} &= \Theta(S^2 S^{*2}\lambda_{(-2,0)} + \lambda_{(2,0)}S^2 S^{*2}).
\end{split}    
\end{equation}
Thus \eqref{eqn:induction_car_sym} together with the identities in \eqref{eqn:elements_B} shows that 
\(\lambda_{(m,0)} S S^{*} \lambda_{(-k,\epsilon)} \in \Theta(\mathcal{F}_2)\) 
for all \(0 \le m, k \le 1\), \(\epsilon \in \mathbb{Z}_2\). 

Now, fix \(n\in \mathbb{N}\), and suppose \(\lambda_{(m,0)} S^n S^{*n} \lambda_{(-k,\epsilon)} \in \Theta(\mathcal{F}_2)\) for all \(0 \le m, k \le 2^n -1\), \(\epsilon \in \mathbb{Z}_2\). Consider \(\lambda_{(m_1 ,0)}S^{n+1}S^{*(n+1)}\lambda_{(-m_2, \epsilon_1)}\in \mathcal{H}\) for some \(0\le m_1, m_2 \le 2^{n+1}-1\) and \(\epsilon_1 \in \mathbb{Z}_2\). 
Then
\begin{align*}
&\lambda_{(m_1 ,0)}S^{n+1}S^{*(n+1)}\lambda_{(-m_2, \epsilon_1)}\\
=&
\begin{cases}
\left( \lambda_{(1,0)} SS^* \right)
S \left( \lambda_{\left( \frac{m_1 - 1}{2}, 0 \right)} S^n S^{*n} \lambda_{\left( -\frac{m_2 - 1}{2}, 0 \right)} \right) S^*
\left( SS^* \lambda_{(-1, \epsilon_1)} \right),\text{ if \(m\), \(n\) odd},\\
\left( \lambda_{(1,0)} SS^* \right)
S \left( \lambda_{\left( \frac{m_1 - 1}{2}, 0 \right)} S^n S^{*n} \lambda_{\left( -\frac{m_2}{2}, 0 \right)} \right) S^*
\left( SS^* \lambda_{(0, \epsilon_1)} \right),\text{ if \(m\) odd, \(n\) even},\\
S \left( \lambda_{\left( \frac{m_1}{2}, 0 \right)} S^n S^{*n} \lambda_{\left( -\frac{m_2 - 1}{2}, 0 \right)} \right) S^*
\left( SS^* \lambda_{(-1, \epsilon_1)} \right),\text{ if \(m\) even, \(n\) odd},\\
S \left( \lambda_{\left( \frac{m_1}{2}, 0 \right)} S^n S^{*n} \lambda_{\left( -\frac{m_2}{2}, 0 \right)} \right) S^*
\left( SS^* \lambda_{(0, \epsilon_1)} \right),
\ \text{ if \(m\), \(n\) even}.
\end{cases}
\end{align*}

Let 
\[
z \coloneqq \lambda_{\left( \frac{m_1 - \mu_1}{2}, 0 \right)} S^n S^{*n} \lambda_{\left( -\frac{m_2 - \mu_2}{2}, 0 \right)}
\]
in any of the cases above with \(\mu_1, \mu_2 \in \mathbb{Z}_2\). Then \( z = \Theta(y) \) for some \( y \in \mathcal{F}_2 \) by the induction hypothesis. Note that 
\begin{equation}\label{eqn:final_induction}
S z S^* = \overline{\chi}(z) SS^* = \Theta(\chi(y) SS^*) \in \Theta(\mathcal{F}_2)
\end{equation}
since \( \chi(y), SS^* \in \mathcal{F}_2 \).

Observe that each term in the right hand side of $\lambda_{(m_1 ,0)}S^{n+1}S^{*(n+1)}\lambda_{(-m_2, \epsilon_1)}$ in the cases above belongs to \( \Theta(\mathcal{F}_2) \) by \eqref{eqn:elements_B} and \eqref{eqn:final_induction}. So 
\[
\lambda_{(m_1, 0)} S^{n+1} S^{*(n+1)} \lambda_{(-m_2, \epsilon_1)} \in \Theta(\mathcal{F}_2).
\]
Therefore, \(\mathcal{H}\subset \Theta(\mathcal{F}_2)\) is proved by induction. 
Consequently, $\mathcal{H} = \Theta(\mathcal{F}_2)$, and $\Theta: \mathcal{F}_2 \to \mathcal{H}$ is a $^*$-isomorphism. 

By the definition of $\Theta$, we have $\Theta \circ \varphi|_{\mathcal{F}_2} = \varphi|_{\mathcal{H}} \circ \Theta$. From the details of the proof of Theorem~\ref{thm:maximality_of_cuntz}, we see that $\varphi|_{\mathcal{F}_2} : \mathcal{F}_2 \to \mathcal{F}_2$ satisfies the Rokhlin property. 
So does $\varphi|_{\mathcal{H}}$ by Lemma~\ref{lem:Roklin_iso}.
\end{proof}

\begin{theo}\label{thm:bunce_sym_roklin}
Let \(\overline{\varphi}\) be the \(^*\)-automorphism of \(\overset{\longrightarrow}{\mathcal{F}}\) induced by the \(^*\)-endomor\-phism \(\varphi: \mathcal{F}\to \mathcal{F}\). Then \(\overline{\varphi}\) has pointwise finite relative Rokhlin dimension with respect to the inclusion \(\overset{\longrightarrow}{\mathcal{H}} \subset \overset{\longrightarrow}{\mathcal{F}}\).
\end{theo}

\begin{proof}
Since the restriction \(\varphi|_{\mathcal{H}}\) has the Rokhlin property by Proposition \ref{prop:rokhlin_property_car_sym}, it follows from \cite[Proposition 2.2]{BH14} that the restriction \(\overline{\varphi}|_{\overset{\longrightarrow}{\mathcal{H}}}\)
of $\overline{\varphi}$ also has the Rokhlin property. In particular, \(\overline{\varphi}|_{\overset{\longrightarrow}{\mathcal{H}}}\) has finite Rokhlin dimension. Let \(\mathcal{F}_0 = \operatorname{span}\left\{ \lambda_{(m, 0)} S^n S^{*n} \lambda_{(k, \epsilon)} : 0 \le m \le 2^n - 1,\; (k, \epsilon) \in G,\; n \in \mathbb{N} \right\}\). Note that \eqref{eqn:bunce_sym_reformulation} implies \(\mathcal{F}_0\) is dense in \(\mathcal{F}\).

Recall that \(\varphi^{(n)} : \mathcal{F} \to \overset{\longrightarrow}{\mathcal{F}}\) is the canonical inductive limit \(^*\)-homomorphism satisfying \(\varphi^{(n+1)} \circ \varphi = \varphi^{(n)}\), and \(\overset{\longrightarrow}{\mathcal{F}} = \overline{\bigcup_n \varphi^{(n)}(\mathcal{F}_0)}\).
Let \(\Omega \subset \overset{\longrightarrow}{\mathcal{F}}\) be a given finite set, 
\(\varepsilon > 0\), and \(0<N \in \mathbb{Z}\). We may assume without loss of generality that \(\Omega \subset \bigcup_n \varphi^{(n)}(\mathcal{F}_0)\). Fix \(p > N\), and choose a finite set \(\Omega_p \subset \mathcal{F}_0\) such that \(\varphi^{(p)}(\Omega_p) = \Omega\). Let \(e_{(n,m)} \coloneqq \lambda_{(m, 0)} S^n S^{*n}\), where \(0 \le m \le 2^n - 1\) and \(n \in \mathbb{N}\). Then \(e_{(n,m)} \in \mathcal{H}\), and any element of the form \(\lambda_{(m, 0)} S^n S^{*n} \lambda_{(k, \epsilon)} \in \mathcal{F}_0\) can be written as \(e_{(n,m)} \lambda_{(k, \epsilon)}\).

Let \(\Omega_p = \{y_s : 1 \le s \le L\}\) where $L=|\Omega_p|(<\infty)$. Then each \(y_s\) can be written as
\begin{equation*}
    y_s = \sum_{i=1}^{N(s)} a_i^s \, e_{(n_i^s, m_i^s)} \, \lambda_{(k_i^s, \epsilon_i^s)},
\end{equation*}
where \(a_i^s \in \mathbb{C}\), and \(e_{(n_i^s, m_i^s)} \in \mathcal{H}\). Then 
\[
\Lambda_p \coloneqq \{ a_i^s e_{(n_i^s, m_i^s)} : 1 \le i \le N(s), \; 1 \le s \le L \}
\quad \text{and} \quad
\Lambda \coloneqq \varphi^{(p)}(\Lambda_p)
\]
are finite subsets of \(\mathcal{H}\) and \(\overset{\longrightarrow}{\mathcal{H}}\), respectively. 

Let \( r \coloneqq \max\{ N(s) : 1 \le s \le L \} \). For a finite set \(\Lambda\), a positive integer \(N\), and \(\varepsilon_1 \coloneqq \frac{\varepsilon}{2r}\), 
by the finite Rokhlin dimension of the \(^*\)-automorphism \(\overline{\varphi} : \overset{\longrightarrow}{\mathcal{H}} \to \overset{\longrightarrow}{\mathcal{H}} \),
there exist positive contractions
\[
\{ f_k^{(l)} : 1 \le k \le N, \; 0 \le l \le d \} \subset \overset{\longrightarrow}{\mathcal{H}}
\]
satisfying the conditions of Definition \ref{defn:Rokhlin_dimension_single}.

Next, we show that 
\(
\{ f_k^{(l)} : 1 \le k \le N, \; 0 \le l \le d \} \subset \overset{\longrightarrow}{\mathcal{H}}
\subset \overset{\longrightarrow}{\mathcal{F}}\)
satisfies the conditions of Definition  \ref{defn:Rokhlin_dim_inclusion}
for \(\Omega\), \(\varepsilon\), and \(N\). In fact, 
\[
\begin{aligned}
&\left\| \left( \sum_{l=0}^d \sum_{k=1}^N f_k^{(l)} \right) \varphi^{(p)}(y_s) - \varphi^{(p)}(y_s) \right\| \\
&\quad= \left\| \sum_{i=1}^{N(s)} \left( \left( \sum_{l=0}^d \sum_{k=1}^N f_k^{(l)} \right) \varphi^{(p)}\left(a_i^s \, e_{(n_i^s, m_i^s)} \, \lambda_{(k_i^s, \epsilon_i^s)}\right) - \varphi^{(p)}\left(a_i^s \, e_{(n_i^s, m_i^s)} \, \lambda_{(k_i^s, \epsilon_i^s)}\right) \right) \right\| \\
&\quad\le \sum_{i=1}^{N(s)} \left\| \left( \sum_{l=0}^d \sum_{k=1}^N f_k^{(l)} \right) \varphi^{(p)}\left(a_i^s \, e_{(n_i^s, m_i^s)}\right) - \varphi^{(p)}\left(a_i^s \, e_{(n_i^s, m_i^s)}\right) \right\| \\
&\quad< N(s) \varepsilon_1 < \varepsilon.
\end{aligned}
\]
Using a similar argument, we see that the other conditions of Definition \ref{defn:Rokhlin_dim_inclusion} hold as well. Hence, \(\overline{\varphi}\) has pointwise finite relative Rokhlin dimension with respect to the inclusion \(\overset{\longrightarrow}{\mathcal{H}} \subset \overset{\longrightarrow}{\mathcal{F}}\).
\end{proof}
\begin{theo}\label{thm:maximality_car_sym}
\(\C \rtimes_{\sigma} \mathbb{Z}_2 \) is a maximal \(C^*\)-subalgebra of \(\Q \rtimes_{\sigma} \mathbb{Z}_2\).
\end{theo}
\begin{proof}
By Proposition \ref{prop:cuntz_inclusion_sym} and Lemma \ref{lem:inclusion_correspondence}, it suffices to show the maximality of the inclusion
\[
\overset{\longrightarrow}{\mathcal{H}} \rtimes_{\overline{\varphi}} \mathbb{Z} \subset \overset{\longrightarrow}{\mathcal{F}} \rtimes_{\overline{\varphi}} \mathbb{Z}.
\]
Using Lemma \ref{lem:inclusion_correspondence} again, together with Theorem~\ref{thm:bunce_sym}, we obtain that \(\overset{\longrightarrow}{\mathcal{H}}\) is a maximal \(C^*\)-subalgebra of \(\overset{\longrightarrow}{\mathcal{F}}\). Since \(\overline{\varphi}\) has pointwise finite relative Rokhlin dimension by Theorem \ref{thm:bunce_sym_roklin}, it follows from Theorem \ref{thm:pointwise_finite_inclusion_criteria} that
\(
\overset{\longrightarrow}{\mathcal{H}} \rtimes_{\overline{\varphi}} \mathbb{Z} 
\)
is a maximal \(C^*\)-subalgebra of \(\overset{\longrightarrow}{\mathcal{F}} \rtimes_{\overline{\varphi}} \mathbb{Z} \).
\end{proof}

\begin{cor}\label{cor:maximality_cuntz_fixed}
Using the notation of Theorem \ref{thm:fixed_sym}, 
the \(C^*\)-algebra \(C^*(T, V)\) is a maximal \(C^*\)-subalgebra of \(C^*(T, V, R)\). Equivalently, \(\mathcal{O}_2^\sigma\) is a maximal \(C^*\)-subalgebra of \(\mathcal{Q}_2^\sigma\).
\end{cor}

\begin{proof}
By \cite[Proposition 3.4]{Rieffel80}, there exists a projection \(e \coloneqq \frac{1}{2}(1+\lambda_{(0,1)}) \in \C \rtimes_{\sigma} \mathbb{Z}_2\) such that the $^*$-isomorphism 
\(e(\Q \rtimes_{\sigma} \mathbb{Z}_2)e \cong \Q^\sigma\) restricts to a $^*$-isomorphism 
\(e(\C \rtimes_{\sigma} \mathbb{Z}_2)e \cong \C^\sigma\).
It is straightforward to see that \(e\) is a full projection in \(\C \rtimes_{\sigma} \mathbb{Z}_2\). By Lemma \ref{lem:inclusion_correspondence} and Theorem~\ref{thm:maximality_car_sym}, \(e(\C \rtimes_{\sigma} \mathbb{Z}_2)e\) is a maximal \(C^*\)-subalgebra of \(e(\Q \rtimes_{\sigma} \mathbb{Z}_2)e\). Hence, \(\mathcal{O}_2^\sigma\) is a maximal \(C^*\)\nobreakdash-subalgebra of \(\mathcal{Q}_2^\sigma\). The remaining conclusion follows from Theorem~\ref{thm:fixed_sym}.
\end{proof}

Let us end the paper by the following remark. 

\begin{rmk}
    For any natural number $n \geq 2$, one can define the $n$-adic ring \(C^*\)-algebra $\mathcal{Q}_n$ by replacing the index 2 with $n$ in the definition of $\mathcal{Q}_2$. The algebra $\mathcal{Q}_n$ contains a canonical copy of the Cuntz algebra $\mathcal{O}_n \coloneqq C^*(s_1, s_2, \ldots, s_n)$, where the generating isometries is defined as

$$
s_k \coloneqq u^{k-1} s_n, \quad k = 1, \ldots, n.
$$

The constructions, results, and proofs from Section \ref{sec:cuntz_inclusion} are expected to carry over with suitable modifications to this more general setting.

Moreover, there exist at least two natural generalized flip-flop \(^*\)-automorphisms on $\mathcal{O}_n$ \cite{AR22}. While not all such \(^*\)-automorphisms may extend to $\mathcal{Q}_n$, it is plausible that some appropriately extendable generalized flip-flops allow for a partial extension of results from Section \ref{sec:cuntz_inclusion_sym}. We leave a detailed study of such extensions and their consequences to future work.
\end{rmk}


\bibliographystyle{elsarticle-harv} 
 \bibliography{mybib}






\end{document}